\documentclass[11pt,a4]{amsart}
\usepackage{amsfonts,amsmath,amssymb}
\usepackage{mathabx}
\usepackage{mathrsfs}
\usepackage[latin1]{inputenc}
\usepackage[usenames,dvipsnames]{pstricks}
\newtheorem{theorem}{Theorem}[section]

\newtheorem{corollary}[theorem]{Corollary}

\newtheorem{definition}[theorem]{Definition}

\newtheorem{lemma}[theorem]{Lemma}

\newcommand{\R}{\mathbb{R}}

\newcommand{\M}{\mathcal{M}}
\newcommand{\g}{\mathrm{g}}

\newcommand{\dd}{\mathrm{d}}
\newcommand{\dv}{\,\mathrm{dv}^{n}}
\newcommand{\dve}{\,\mathrm{dv}^{n}}

\newcommand{\dvv}{\,\mathrm{dv}^{2n-1}}
\newcommand{\ds}{\,\mathrm{d\sigma}^{n-1}}
\newcommand{\dse}{\,\mathrm{d\sigma}^{n-1}}

\newcommand{\dss}{\,\mathrm{d\sigma}^{2n-2}}
\newcommand{\I}{\mathcal{I}}
\newcommand{\h}{\mathcal{H}}

\newcommand{\p}{\partial}
\newcommand{\norm}[1]{\left\Vert#1\right\Vert}
\newcommand{\abs}[1]{\left|#1\right|}
\newcommand{\set}[1]{\left\{#1\right\}}
\newcommand{\para}[1]{\left(#1\right)}
\newcommand{\cro}[1]{\left[#1\right]}

\newcommand{\seq}[1]{\left<#1\right>}

\newcommand{\To}{\longrightarrow}
\usepackage{graphicx}
\newcommand{\dive}{\textrm{div}}

\usepackage{times}

\begin{document}

\title[Magnetic Schr\"odinger equation]
{Stable determination of coefficients in the dynamical Schr\"odinger equation in a magnetic field}
\author[M.~Bellassoued]
{Mourad~Bellassoued}

\address{M.~Bellassoued.  University of Tunis El Manar, National Engineering School of Tunis, ENIT-LAMSIN, B.P. 37, 1002 Tunis, Tunisia}
\email{mourad.bellassoued@fsb.rnu.tn}

\date{\today}

\maketitle

\begin{abstract}
In this paper we consider the inverse problem of determining on a compact Riemannian manifold
the electric potential or the magnetic field in a Schr\"odinger equation
 with Dirichlet data from measured Neumann boundary observations. This information is enclosed in the dynamical
Dirichlet-to-Neumann map associated to the magnetic
Schr\"odinger equation. We prove in dimension $n\geq 2$ that the
knowledge of the Dirichlet-to-Neumann map for the
Schr\"odinger equation uniquely determines
the magnetic field and the electric potential and we establish H\"older-type stability.\\
{\bf Keywords:} Stability estimates, magnetic Schr\"odinger inverse problem, Dirichlet-to-Neumann map.

\end{abstract}


\section{Introduction and main results}

This article is devoted to the  study of the following inverse boundary value problem:  given a Riemannian manifold with boundary
determine the magnetic potential in a dynamical Schr\"odinger equation in a magnetic field  from the
observations made at the boundary. Let $(\M,\,\g)$ be a smooth and compact
Riemannian manifold with boundary $\partial \M$. We denote by $\Delta$ the Laplace-Beltrami operator associated to the Riemannian metric $\g$. In local coordinates,
$\g(x)=(\g_{jk})$, the Laplace operator $\Delta$ is given by
$$
\Delta=\frac{1}{\sqrt{\abs{\g}}}\sum_{j,k=1}^n\frac{\p}{\p
x_j}\para{\sqrt{\abs{\g}}\,\g^{jk}\frac{\p}{\p x_k}}.
$$
Here $(\g^{jk})$ is the inverse of the metric $\g$ and $\abs{
\g}=\det(\g_{jk})$. 
In this paper we study an inverse problem for the
dynamical Schr\"odinger equation in the presence of a magnetic
potential. Given $T>0$, we denote $Q=(0,T)\times\M$ and $\Sigma=(0,T)\times\p\M$. We consider the
following initial boundary value problem for the magnetic Schr\"odinger
equation with a magnetic potential $A$ and electric potential $V$,
\begin{equation}\label{1.2}
\left\{
\begin{array}{llll}
\para{i\partial_t+\h_{A,V}}u=0  & \textrm{in }\; Q,\cr
u(0,\cdot )=0 & \textrm{in }\; \M ,\cr
u=f & \textrm{on } \; \Sigma,
\end{array}
\right.
\end{equation}
where
\begin{equation}\label{H}
\h_{A,V}=\frac{1}{\sqrt{\abs{\g}}}\sum_{j,k=1}^n\para{\frac{\p}{\p
x_j}-ia_j}\sqrt{\abs{\g}}\,\g^{jk}\para{\frac{\p}{\p x_k}-ia_k}+V
=\Delta-2i\,A\cdot\nabla-i\,\delta A+|A|^2+V.
\end{equation}
Here $V:\M\to\R$ is real valued function is the electric potential and $A=a_jdx^j$ is a covector field (1-form) with real-valued coefficients $a_j\in \mathcal{C}^{\infty}(\M)$ is the magnetic potential and $\delta$ is the coderivative (codifferential) operator sending $1$-forms  to a function by the formula 
$$
\delta A=\frac{1}{\sqrt{\abs{\g}}}\sum_{j,k=1}^n \frac{\p}{\p x^j}\para{\g^{jk}\sqrt{\abs{\g}} a_k}.
$$
 We may define the Dirichlet to Neumann (DN) map associated with magnatic Schr\"odinger operator $\h_{A,V}$ by
\begin{equation}\label{1.3}
\Lambda_{A,V}(f)=(\p_\nu+i\,A(\nu))u,\quad f\in H^{2,1}(\Sigma),
\end{equation}
where $\nu=\nu(x)$ denotes the unit outward normal to $\p\M$ at $x$ and $H^{2,1}(\Sigma)$ is anisotropic Sobolev space defined below.
\medskip

We consider the inverse problem to know  whether the DN map $\Lambda_{A,V}$ determines uniquely the magnetic potential $A$ and the electric potential $V$.
\medskip

In the absence of the magnetic potential $A$, the identifiability problem of the electric potential $V$ was solved by \cite{[BellDSSF-2]}. In the presence of a magnetic potential $A$, let us observe that there is an obstruction to uniqueness. In fact as it was noted in \cite{[Eskin]}, the DN map is invariant under the gauge transformation of the magnetic potential. Namely, given $\varphi\in
\mathcal{C}^1(\overline{\M})$ such that $\varphi|_{\p\M}=0$ one has
\begin{equation}\label{1.4}
e^{-i\varphi}\h_{A,V} e^{i\varphi}=\h_{A+d\varphi,V},\quad
e^{-i\varphi}\Lambda_{A,V} e^{i\varphi}=\Lambda_{A+d\varphi,V}=\Lambda_{A,V},\quad d\varphi=\sum_{j=1}^n\frac{\p\varphi}{\p x^j}dx^j.
\end{equation}
Therefore, the magnetic potential $A$  cannot be uniquely determined by
the DN map $\Lambda_{A,V}$. From a geometric view point this can be seen as follows. 
Since $\M$ is a compact Riemannian manifold with boundary, then for every covector $A \in H^k(\M,T^*\M)$, there exist uniquely determined
$A^s \in H^k(\M,T^*\M)$ and $\varphi \in H^{k+1}(\M)$ such that:
$$
A=A^s+d\varphi,\quad \delta A^s=0,\quad \varphi|_{\p\M}=0.
$$
 We call the fields $A^s$ and $d\varphi$ the solenoidal and potential parts of the covector $A$.
The non-uniqueness manifested in (\ref{1.4}) says that the best we could hope to reconstruct from
the DN map $\Lambda_{A,V}$ is the solenoidal part $A^s$ of the covector $A$.
\medskip

Physically, our inverse problem consists in determining the magnetic field $A^s$ induced by the magnetic potential $A$ of an anisotropic medium by probing it with disturbances generated on the boundary. The data are responses of the medium to these disturbances which are measured on the boundary and the goal is to recover the magnetic field $A^s$ which describes the property of the medium. Here we assume that the medium is quiet initially and $f$ is a disturbance which is used to probe the medium. Roughly speaking, the data is $(\p_\nu+i\nu\cdot A) u$ measured on the boundary for different choices of $f$.
\medskip

The uniqueness in the determination of electromagnetic potential, appearing in a Sch\"odinger equation in a domain with obstacles, 
from the DN map was proved by Eskin \cite{[Eskin]}. The main ingredient in his proof is the construction of geometric optics solutions. Using this geometric optics construction Salazar \cite{[Salazar]} shows that the boundary data allows us to recover integrals of the potentials along \textit{light rays} and he establish the uniqueness of these potentials
modulo a gauge transform. Also, a logarithmic stability estimate is
obtained and the presence of obstacles inside the domain is studied. In 
\cite{[Avdonin-al]}, Avdonin and al use the so-called BC (boundary control) method to  prove that the DN map  determines the electrical potential in a one dimensional Schr\"odinger equation.
\medskip

In recent years significant progress has been made for the problem of identifying the electrical potential. In \cite{[Rakesh-Symes]}, Rakesh and Symes prove that the DN map determines uniquely the time-independent potential in a wave equation. Ramm and Sj\"ostrand
\cite{[Ramm-Sjostrand]} has extended the result in
\cite{[Rakesh-Symes]} to the case of time-dependent potentials.
Isakov \cite{[Isakov1]} has considered the simultaneous
determination of a zeroth order coefficient and a damping
coefficient.  A key ingredient in the existing results is the
construction of complex geometric optics solutions of the wave
equation, concentrated along a line, and the relationship between
the hyperbolic DN map and the $X$-ray transform
play a crucial role. For the wave equation with a lower order term $q(t, x)$, Waters \cite{[Waters]} proves that we can recover the $X$-ray transform of time dependent potentials $q(t, x)$ from
the dynamical Dirichlet-to-Neumann map in a stable way. He derive
conditional H\"older stability estimates for the X-ray transform of $q(t, x)$.
\medskip

 The uniqueness by a local DN map is well solved
(e.g., Belishev \cite{[1]}, Eskin
\cite{[Eskin]}, \cite{[Eskin2]}, Katchlov, Kurylev and Lassas \cite{[KKL]},
Kurylev and Lassas \cite{[KL]}). The stability estimates in the case where the DN map is
considered on the whole lateral boundary were
established in, Stefanov and
Uhlmann \cite{[SU]}, Sun \cite{[Sun]}, Bellassoued ans Dos Santos Ferriera \cite{[BellDSSF]}. 
In \cite{[Montalo]}  C.Montalo proves H\"older type stability estimates near generic simple Riemannian metrics for the inverse problem of recovering simultaneously the metric, the magnetic field, and electric potential  from the associated hyperbolic Dirichlet to Neumann (DN) map  modulo a class of gauge transformations.
\medskip

In the case of the Schr\"odinger equation,  Avdonin and Belishev  gave an affirmative answer to this question for smooth metrics conformal to the
Euclidean metric in \cite{[AB]}. Their approach is based on the boundary control method introduced by Belishev \cite{[1]} and uses in an essential way a unique continuation property. Because of the use of this qualitative property, it seems unlikely that the boundary control method would provide
accurate stability estimates. More precisely, when $\M$ is a bounded domain of $\R^n$, and $\varrho,\,q\in\mathcal{C}^2(\overline{\M})$ are real functions, Avdonin and Belishev \cite{[AB]} show that for any fixed $T>0$ the response operator (or the Neumann-to-Dirichlet map) of the Schr\"odinger equation $(i\varrho\p_tu+\Delta u-qu)=0$ uniquely determines the coefficients $\varrho$ and $q$. The problem is reduced to recovering $\varrho,\,q$ from the boundary
spectral data. The spectral data are extracted from the response operator by the use of a variational principle.
\medskip

The analogue problem for the wave equation has a long history. Unique determination of the metric goes back to Belishev and Kurylev \cite{[BK]} using
the boundary control method and involves works of Katchlov, Kurylev and Lassas \cite{[KKL]}, Kurylev and Lassas \cite{[KL]}, Lassas and Oksanen \cite{[LO]} and Anderson, Katchalov,
Kurylev, Lassas and Taylor \cite{[AKKLT]}. In fact, Katchalov, Kurylev, Lassas and Mandache proved that the determination of the metric from the
Dirichlet-to-Neumann map was equivalent for the wave and Schr\"odinger equations (as well as other related inverse problems) in \cite{[KKLM]}.
\medskip

\medskip

The importance of control theory for inverse problems was first understood by Belishev \cite{[1]}. He used control theory to develop the first variant of the control (BC) method. This method gives an efficient way to reconstruct a Riemannian manifold via its response operator (dynamical Dirichlet-to-Neumann map) or spectral data (a spectrum of the Beltrami-Laplace operator and traces of normal derivatives of the eigenfunctions), themselves, whereas the coefficients on these manifolds are recovered automatically. More precisely, let $(\M,\g)$ and $(\M',\g')$ be two smooth compact manifolds with mutual boundary $\p\M=\p\M'=\Gamma$ endowed with smooth potentials $q$ and $q'$ respectively, $\Lambda_{\g,q}$ and $\Lambda_{\g',q'}$ their DN-map on $(0,T)\times\p\M$, if $\Lambda_{\g,q}=\Lambda_{\g',q'}$ then there exists a diffeomorphism $\Psi:\M\To\M'$ such that $\Psi_{|\Gamma}=id$, $\g=\Psi^*\g'$, and $q=q'\circ\Psi$.
\medskip

As for the stability of the wave equation in the Euclidian case, we also refer to \cite{[Sun]} and \cite{[IS]}; in those papers, the Dirichlet-to-Neumann map was considered on the  whole boundary. Isakov and Sun \cite{[IS]} proved that the difference in some subdomain of two coefficients is estimated by an operator norm of the difference of the corresponding local Dirichlet-to-Neumann maps, and that
the estimate is of H\"older type. Bellassoued, Jellali and Yamamoto \cite{[Bell-Jel-Yama2]} considered the inverse problem of recovering a
time independent potential in the hyperbolic equation from the partial Dirichlet-to-Neumann map. They proved a logarithm stability estimate. Moreover in \cite{[Rakesh1]}  it is proved that if an unknown coefficient belongs to a given finite dimensional vector space, then the uniqueness
follows by a finite number of measurements on the whole boundary. In \cite{[Bellassoued-Benjoud]}, Bellassoued and Benjoud used complex geometrical
optics solutions concentring near lines in any direction to prove that the Dirichlet-to-Neumann map determines uniquely the magnetic field induced by
a magnetic potential in a magnetic wave equation.
\medskip

In the case of the anisotropic wave equation, the problem of establishing stability estimates in determining the metric was studied by Stefanov and
Uhlmann in \cite{[SU], [SU2]} for metrics close to Euclidean and generic simple metrics. In \cite{[BellDSSF]}, the author and Dos Santos Ferriera proved stability estimates for the wave equation in determining a conformal factor close to 1 and time independent potentials in simple geometries. We refer to this paper for a longer bibliography in the case of the wave equation. In \cite{[LiuO]} Liu and Oksanen  consider the problem to reconstruct a wave speed $c$ from acoustic boundary measurements modelled by the hyperbolic Dirichlet to Neumann map. They introduced a reconstruction formula for $c$ that is based on the Boundary Control method and incorporates features also from the complex geometric optics solutions approach.
\medskip

For the DN map for an elliptic equation, the paper
by Calder\'on \cite{[Calderon]} is a pioneering work. We also refer to
Bukhgeim and Uhlamnn \cite{[Bukhgeim-Uhlmann]}, Hech-Wang
\cite{[Hech-Wang]}, Salo \cite{[Salo]} and Uhlmann \cite{[Uhlmann]}
as a survey. In \cite{[DKSU]} Dos Santos
Ferreira, Kenig, Sjostrand, Uhlmann prove that the knowledge of the
Cauchy data for the Schr\"odinger equation in the presence of
magnetic potential, measured on possibly very small subset of the
boundary, determines uniquely the magnetic field. In \cite{[Tzou]},
Tzou proves a log log-type estimate which show that the
magnetic field and the electric potential of the magnetic
Schr\"odinger equation depends stably on the DN
map even when the boundary measurement is taken only on a subset
that is slightly larger than the half of the boundary. In \cite{[CY]},
Cheng and Yamamoto prove that the stability estimation imply the
convergence rate of the Tikhonov regularized solutions.
\medskip

The main goal of this paper is to study the stability of the inverse problem for the dynamical anisotropic Schr\"odinger equation with magnetic and electric potentials.
We follow the same strategy as in \cite{[BellDSSF]} inspired by the works of Dos Santos Ferreira, Kenig, Salo and Uhlmann \cite{[DKSU]},
Stefanov and Uhlmann \cite{[SU],[SU2]} and Bellassoued and Choulli \cite{[Bel-Choul]}.

\medskip
In the present paper, we prove a H\"older-type estimate which shows that a
magnetic field $A^s$ induced by a magnetic potential and the electric potential depends stably on the DN map $\Lambda_{A,V}$.
\subsection{Notations and well-posedness of the magnetic Schr\"odinger equation}
First, we will consider the initial-boundary value problem for the magnetic Schr\"odinger equation on a manifold with boundary (\ref{1.2}). This initial
boundary value problem corresponds to an elliptic operator $-\h_{A,V}$ given by (\ref{H}). In appendix A we develop an invariant approach to prove existence and uniqueness of solutions and to study their regularity proprieties.
\medskip

Before stating our first main result, we recall the following preliminaries. We refer to \cite{[Jost]} for the differential calculus of tensor fields on a Riemannian manifold. Let $(\M,\g)$ be an $n$-dimensional, $n\geq 2$, compact Riemannian manifold, with smooth boundary and smooth metric $\g$. Fix a coordinate system $x=\para{x^1,\ldots,x^n}$ and let $\para{\p_1,\dots,\p_n}$ be the corresponding tangent vector fields. For $x\in \M$, the inner product and the norm on the tangent space $T_x\M$ are given by
\begin{gather*}
\g(X,Y)=\seq{X,Y}=\sum_{j,k=1}^n\g_{jk}X^jY^k, \\
\abs{X}=\seq{X,X}^{1/2},\qquad  X=\sum_{i=1}^nX^i\p_i,\quad Y=\sum_{i=1}^n
Y^i\p_i.
\end{gather*}
The cotangent space $T_x^*\M$ is the space of linear functionals on $T_x\M$. Its elements are
called covectors or one-forms. The disjoint union of the tangent spaces $T\M=\cup_{x\in\M} T_x\M$ is called the tangent bundle of $\M$. Respectively, the cotangent bundle $T^*\M$ is the union of the spaces $T^*_x\M$, $x\in\M$. A $1$-form $A$ on the manifold $\M$ is a function that assigns to each point $x\in\M$ a covector $A(x)\in T^*_x\M$. 
\medskip

An example of a one-form is the differential of a function $f \in\mathcal{C}^\infty(\M)$, which is defined by
$$
df_x(X)=\sum_{j=1}^n X^j\frac{\p f}{\p x_j},\quad X=\sum_{j=1}^nX^j\p_j.
$$
Hence $f$ defines the mapping $df : T\M\to\R$, which is called the differential of $f$ given by
$$
df(x,X)=df_x(X).
$$
In local coordinates,
$$
df=\sum_{j=1}^n\p_jfdx^j.
$$
The Riemannian metric $\g$ induces a natural isomorphism $\imath : T_x\M\to T^*_x\M$ given
by $\iota(X) = \seq{X,\cdot}$. For $X\in T_x\M$ denote $X^\flat= \imath(X)$, and similarly for $A\in T^*_x\M$ we denote $A^\sharp=\imath^{-1}(A)$, $\imath$ and $\imath^{-1}$ are called musical isomorphisms. The \textit{sharp} operator is given by
\begin{equation}\label{1.5}
T^*_x\M\To T_x\M,\quad A\longmapsto A^\sharp,
\end{equation}
given in local coordinates by
\begin{equation}\label{1.6}
(a_jdx^j)^\sharp=a^j\p_j,\quad a^j=\sum_{k=1}^n\g^{jk}a_k,
\end{equation}
where $(dx^1,\dots,dx^n)$ is the basis in the space $T^*_x\M$ which is the dualto the basis $(\p_1,\dots,\p_n)$. For the Riemannian manifold $(\M,\g)$ we define the inner product of $1$-forms in $T^*_x\M$ by 
\begin{equation}\label{1.7}
\seq{A,B}=\seq{A^\sharp,B^\sharp}=\sum_{j,k=1}^n \g^{jk}a_jb_k=\sum_{j,k=1}^n \g_{jk}a^jb^k.
\end{equation}
The metric tensor $\g$ induces the Riemannian volume $\dv=|\g|^{1/2}\dd x_1\wedge\cdots \wedge \dd x_n$. We denote by $L^2(\M)$ the completion
of $\mathcal{C}^\infty(\M)$ endowed with the usual inner product
$$
\para{f_1,f_2}=\int_\M f_1(x) \overline{f_2(x)} \dv_{\g},\qquad  f_1,f_2\in\mathcal{C}^\infty(\M).
$$
A smooth section of vector bundle $E$ over the Rieamannian manifold $\M$ is a smooth map $\mathfrak{s}:\M\to E$ such that for each $x\in\M$, $\mathfrak{s}(x)$ belongs to the fiber over $x$. We denote by $\mathcal{C}^\infty(\M,E)$ the space of smooth sections of the vector bundle $E$. Using this, we denote $\mathcal{C}^\infty(\M,T\M)$ the space of smooth vector fields on $\M$ and $\mathcal{C}^\infty(\M,T^*\M)$ the space of smooth $1$-forms on $\M$. Similarly, we may define the spaces $L^2(\M,T^*\M)$ (resp. $L^2(\M,T\M)$) of square integrable $1$-forms (resp. vectors) by using the inner product
\begin{equation}\label{1.8}
\para{A,B}=\int_\M\seq{A,\overline{B}}\dv,\quad A,B \in T^*\M.
\end{equation}
Let $T^k_x\M$ be the space of tensors fields of type $k$ on $T^k_x\M$. We denote by $T^k\M$ the tensor bundle of type $k$. In the local coordinate system a $k$-tensor field $\mathfrak{u}$ can be written as
$$
\mathfrak{t}=t_{j_1,\dots,j_k}dx^{j_1}\otimes\dots\otimes dx^{j_k}.
$$
For each $x\in\M$, $T_x\M$ is endowed with an inner product as follows
$$
\seq{\mathfrak{t}_1 ,\mathfrak{t}_2}=\sum_{j_1,\dots,j_k=1}^n \mathfrak{t}_1(\p_{j_1},\dots,\p_{j_k})\mathfrak{t}_2(\p_{j_1},\dots,\p_{j_k}).
$$
Let $\mathcal{C}^\infty(\M,T^k\M)$ the space of the smooth $k$-tensor fields on $\M$. In view of (\ref{1.8}), we denote by $L^2(\M,T^k\M)$ the space of square integrable $k$-tensors fields on $\M$ as the completion of $\mathcal{C}^\infty(\M,T^k\M)$ endowed with the following inner product
$$
\para{\mathfrak{t}_1,\mathfrak{t}_2}=\int_\M\seq{\mathfrak{t}_1,\overline{\mathfrak{t}_2}}\dv,\quad \mathfrak{t}_1,\mathfrak{t}_2 \in T^k\M.
$$
The Sobolev space $H^k(\M)$ is the completion of $\mathcal{C}^\infty(\M)$ with respect to the norm $\norm{\,\cdot\,}_{H^k(\M)}$,
$$
\norm{f}^2_{H^k(\M)}=\norm{f}^2_{L^2(\M)}+\sum_{k=1}^n\|\nabla^k f\|^2_{L^2(\M,T^k\M)}.
$$
where $\nabla^k$ is the covariant differential of $f$ in the metric $\g$.
If $f$ is a $\mathcal{C}^\infty$ function on $\M$, then $\nabla f$ is the vector field such that
$$
X(f)=\seq{\nabla f,X},
$$
for all vector fields $X$ on $\M$. This reads in coordinates
\begin{equation}\label{1.9}
\nabla f=\sum_{i,j=1}^n\g^{ij}\frac{\p f}{\p x_i}\p_j=(df)^\sharp.
\end{equation}
The normal derivative is
\begin{equation}\label{1.10}
\p_\nu u:=\seq{\nabla u,\nu}=\sum_{j,k=1}^n\g^{jk}\nu_j\frac{\p u}{\p x_k},
\end{equation}
where $\nu$ is the unit outward vector field to $\p \M$.
\medskip

Likewise, we say that $1$-form $A=a_jdx^j$ in $H^k(\M,T^*\M)$ if each component $a_j$ in $H^k(\M)$, which can be viewed as the Hilbert space with respect to the norm
$$
\norm{A}_{H^k(\M,T^*\M)}=\sum_{j=1}^n\norm{a_j}_{H^k(\M)}.
$$
Before stating our main results on the inverse problem, we give the following result concerning the wellposedness of  the initial boundary problem (\ref{1.2}), when $u$ is
a weak solution in the class $\mathcal{C}^1(0,T;H^1(\M))$. The following theorem gives conditions on $f$, $A$ and $V$, which
guarantee uniqueness and continuous dependence on the data of the solutions of the magnetic Schr\"odinger equation (\ref{1.2}) with non-homogenous Dirichlet
boundary condition.\\

We denote $\Omega^2(\M)$ the vector space of smooth $2$-forms on $\M$. In local coordinates $2$-form $\omega$ can be represented as
$$
\omega=\omega_{jk}dx^j\wedge dx^k,
$$
where $\omega_{jk}$ are smooth real-valued functions on $\M$. For smooth and compactly supported $2$-form $\omega$ in $\M$, we define the Sobolev norm $H^s(\M,\Omega^2(\M))$, $s\in\R$, by
$$
\norm{\omega}_{H^s(\M,\Omega^2(\M))}=\sum_{jk}\norm{\omega_{jk}}_{H^s(\M)}.
$$
Finally, we introduce the anisotropic Sobolev spaces
$$
H^{2,1}(\Sigma)=H^2(0,T;L^2(\p\M))\cap L^2(0,T;H^1(\p\M)),
$$
equipped with the norm
$$
\norm{f}_{H^{2,1}(\Sigma)}=\norm{f}_{H^2(0,T;L^2(\p\M))}+\norm{f}_{L^2(0,T;H^1(\p\M))}.
$$
Finally we set
$$
H_0^{2,1}(\Sigma)=\set{f\in H^{2,1}(\Sigma),\,f(0,\cdot)=\p_tf(0,\cdot)\equiv 0 }.
$$
\begin{theorem}\label{Th0}
Let $T>0$ be given, $A\in \mathcal{C}^1(\M,T^*\M)$ and $V\in W^{1,\infty}(\M)$. Suppose that $f\in H^{2,1}_0(\Sigma)$. Then the unique solution $u$ of \eqref{1.2} satisfies
\begin{equation}\label{1.11}
u\in\mathcal{C}^1(0,T;H^1(\M)).
\end{equation}
Furthermore we have $\p_\nu u\in L^2(\Sigma)$ and there is a constant $C=C(T,\M,\norm{A}_{W^{1,\infty}},\norm{V}_{W^{1,\infty}})>0$ such that
\begin{equation}\label{1.12}
\norm{\p_\nu u}_{L^2(\Sigma)}\leq C\norm{f}_{H^{2,1}(\Sigma)}.
\end{equation}
The Dirichlet-to-Neumann map  $\Lambda_{A,V}$ defined by (\ref{1.3}) is therefore continuous and we denote by
$\norm{\Lambda_{A,V}}$ its norm in $\mathscr{L}(H^{2,1}(\Sigma),L^2(\Sigma))$.
\end{theorem}
Theorem \ref{Th0} gives a rather comprehensive treatment of the regularity problem for \eqref{1.2} with stronger boundary condition $f$. Moreover,
our treatment clearly shows that a regularity for $f\in H_0^{2,1}(\Sigma)$ is sufficient to obtain the desired interior regularity of $u$ on
$Q$ while the full strength of the assumption $f\in H_0^{2,1}(\Sigma)$ is used to obtain the desired boundary regularity for
$\p_\nu u$ and then the continuity of the Dirichlet-to-Neumann map $\Lambda_{A,V}$.
\subsection{Stable determination}
In this section we state the main stability results. Let us first
introduce the admissible class of manifolds for which we can prove
uniqueness and stability results in our inverse problem. For this we
need the notion of simple manifolds \cite{[SU2]}.
\medskip

Let $(\M,\g)$ be a Riemannian manifold with boundary $\p\M$, we denote by $D$ the Levi-Civita connection on $(\M,\g)$.
For a point $x \in \p\M$, the second quadratic
form of the boundary
$$
\Pi(\theta,\theta)=\seq{D_\theta\nu,\theta},\quad \theta\in T_x(\p\M)
$$
is defined on the space $T_x(\p\M)$. We say that the boundary is strictly convex if the form is positive-definite for all $x \in \p\M$ (see \cite{[Sh]}).
\begin{definition}
We say that the Riemannian manifold $(\M,\g)$ (or that the metric $\g$) is simple in $\M$, if $\p \M$ is
strictly convex with respect to $\g$, and for any $x\in \M$, the exponential map
$\exp_x:\exp_x^{-1}(\M)\To \M$ is a diffeomorphism. The latter means that every two points $x; y \in \M$ are joined by a unique geodesic smoothly
depending on $x$ and $y$.
\end{definition}
Note that if $(\M,\g)$ is simple, one can extend it to a simple manifold $\M_{1}$ such that $\M_1^{\textrm{int}}\supset\M$.
\medskip

Let us now introduce the admissible sets of magnetic potentials $A$ and electric potentials $V$. Let $m_1,m_2>0$ and $k\geq 1$ be given, set
\begin{equation}\label{1.13}
\mathscr{A}(m_1,k)= \set{A\in\mathcal{C}^\infty(\M,T^*\M),\,\,\norm{A}_{H^k(\M,T^*\M)}\leq m_1}.
\end{equation}
and 
\begin{equation}\label{1.14}
\mathscr{V}(m_2)= \set{V\in W^{1,\infty}(\M),\,\,\norm{V}_{W^{1,\infty}(\M)}\leq m_2}.
\end{equation}
The main results of this paper are as follows.
\begin{theorem}\label{Th2}
Let $(\M,\g)$ be a simple compact Riemannian manifold with boundary of dimension $n \geq 2$ and let $T>0$.
There exist $k\geq 1$, $\varepsilon>0$, $C>0$ and $\kappa\in(0,1)$ such that for any $A_1,A_2\in\mathscr{A}(m_1,k)$ and $V_1,V_2\in\mathscr{V}(m_2)$ coincide near the boundary $\p\M$ and any  with $\norm{A_1^s-A_2^s}_{\mathcal{C}^0}\leq \varepsilon$, the following estimate holds true
\begin{equation}\label{1.15}
\norm{A_1^s-A_2^s}_{L^2(\M,T^*\M)}+\norm{V_1-V_2}_{L^2(\M)}\leq C\norm{\Lambda_{A_1,V_1}-\Lambda_{A_2,V_2}}^\kappa
\end{equation}
where $C$ depends on $\M$, $m_1,m_2$, $n$, and $\varepsilon$.
\end{theorem}
For $1$-form $A=a_jdx^j$ where $a_j$ are smooth functions on $\M$. The exterior derivative of $A$ is given by 
$$
dA=\sum_{j,k=1}^n \frac{1}{2}(\p_ja_k-\p_ka_j)dx^j\wedge dx^k
$$ 
where $\wedge$ is the antisymmetric wedge product $dx^j\wedge dx^k=-dx^k\wedge dx^j$. Since $d^2=0$ for all forms, we get
$$
dA=dA^s.
$$
By Theorem \ref{Th2}, we can readily derive the following
\begin{corollary}\label{C1}
Let $(\M,\g)$ be a simple compact Riemannian manifold with boundary of dimension $n \geq 2$ and let $T>0$.
There exist $k\geq 1$, $\varepsilon>0$, $C>0$ and $\kappa\in(0,1)$ such that for any $A_1,A_2\in\mathscr{A}(m_1,k)$ and $V_1,V_2\in\mathscr{V}(m_2)$  coincide near the boundary $\p\M$ and any  with $\norm{A_1^s-A_2^s}_{\mathcal{C}^0(\M)}\leq \varepsilon$, the following estimate holds true
\begin{equation}\label{1.16}
\norm{dA_1-dA_2}_{H^{-1}(\M,\Omega^2(\M))}+\norm{V_1-V_2}_{L^2(\M)}\leq C\norm{\Lambda_{A_1,V_1}-\Lambda_{A_2,V_2}}^{\kappa}
\end{equation}
where $C$ depends on $\M$, $m_1,m_2$, $n$, and $\varepsilon$.
\end{corollary}
By Theorem \ref{Th2}, we can readily derive the following uniqueness result
\begin{corollary}
Let $(\M,\g)$ be a simple compact Riemannian manifold with boundary of dimension $n \geq 2$ and let $T>0$.
There exist $k\geq 1$, $\varepsilon>0$, such that for any $A_1,A_2\in\mathscr{A}(m_1,k)$ and any $V_1,V_2\in\mathscr{V}(m_2)$ coincide near the boundary with $\norm{A_1^s-A_2^s}_{\mathcal{C}}\leq \varepsilon $, we have that $\Lambda_{A_1,V_1}=\Lambda_{A_2,V_2}$ implies
$A_1^s =A_2^s$ and $V_1=V_2$ everywhere in $\M$.
\end{corollary}
\medskip

Our proof is inspired by techniques used by Stefanov and Uhlmann \cite{[SU2]},  and Bellassoued- Dos Santos Ferreira \cite{[BellDSSF-2]} which prove uniqueness theorems for an inverse problem without magnetic potential.
\medskip

The outline of the paper is as follows. In section 2  we study the  geodesical ray transform for $1$-one forms and functions on a manifold. In section 3 we construct special geometrical optics solutions to magnetic Schr\"odinger equations. In section 4 and 5, we establish stability estimates for the solenoidale part of the magntic field and the electric potential. The appendix A is devoted to the study the Cauchy problem for the Schr\"odinger equation and we prove Theorem \ref{Th0}. 
%

\section{Geodesical ray transform on a simple manifold}
\setcounter{equation}{0}
In this section we first collect some  formulas needed in the rest of this paper and introduce the geodesical ray transform for $1$-form. Denote by $\dive X$ the divergence of a vector field $X\in H^1(\M,T\M)$ on $\M$, i.e. in local coordinates (see pp. 42, \cite{[KKL]}),
\begin{equation}\label{2.1}
\dive X=\frac{1}{\sqrt{\abs{\g}}}\sum_{i=1}^n\p_i\para{\sqrt{\abs{\g}}\,X^i},\quad X=\sum_{i=1}^nX^i\p_i.
\end{equation}
Using the inner product of $1$-form, we can define the coderivature operator $\delta$ as the adjoint of the exterior derivative via the relation
\begin{equation}\label{2.2}
\para{\delta A,v}=\para{A,dv},\quad A\in\mathcal{C}^\infty(M,T^*M),\,v\in \mathcal{C}^\infty(\M).
\end{equation}
Then $\delta A$ is related to the divergence of vector fields by $\delta A=-\dive(A^\sharp)$, where the divergence is given by (\ref{2.1}).
If $X\in H^1(\M,T\M)$ the divergence formula reads
\begin{equation}\label{2.3}
\int_\M\dive X \dv=\int_{\p \M}\seq{X,\nu} \ds,
\end{equation}
and for a function $f\in H^1(\M)$ Green's formula reads
\begin{equation}\label{2.4}
\int_\M\dive X\,f\dv=-\int_\M\seq{X,\nabla f} \dv+\int_{\p \M}\seq{X,\nu} f\ds.
\end{equation}
Then if $f\in H^1(\M)$ and $w\in H^2(\M)$, the following identity holds
\begin{equation}\label{2.5}
\int_\M\Delta w f\dv=-\int_\M\seq{\nabla w,\nabla f} \dv+\int_{\p \M}\p_\nu w f \ds.
\end{equation}
For $x\in \M$ and $\theta\in T_x\M$ we denote by $\gamma_{x,\theta}$ the unique geodesic starting at the point $x$ in the direction $\theta$.  We
consider
\begin{align*}
S\M=\set{(x,\theta)\in T\M;\,\abs{\theta}=1}, \quad
S^*\M=\set{(x,p)\in T^*\M;\,\abs{p}=1},
\end{align*}
the sphere bundle and co-sphere bundle of $\M$. The exponential map $\exp_x:T_x\M\To \M$ is given by
\begin{equation}\label{2.6}
\exp_x(v)=\gamma_{x,\theta}(\abs{v}),\quad \theta=\frac{v\,\,}{\abs{v}}.
\end{equation}
A compact Riemannian manifold $(\M,\, \g)$ with boundary is called a convex non-trapping
manifold, if it satisfies two conditions:
\begin{enumerate}
    \item[(i)] the boundary $\p \M$ is strictly convex, i.e., the second fundamental form of the boundary is positive definite at every
    boundary point,
    \item[(ii)] for each $(x,\theta) \in S\M$, the maximal
    geodesic $\gamma_{x,\theta}(t)$ satisfying the initial conditions $\gamma_{x,\theta}(0) = x$ and $\dot{\gamma}_{x,\theta}(0) = \theta$ is defined on
    a finite segment $[\tau_{-}(x,\theta), \tau_{+}(x,\theta)]$. We recall that a geodesic $\gamma: [a, b] \To M$ is maximal
    if it cannot be extended to a segment $[a-\varepsilon_1, b+\varepsilon_2]$, where $\varepsilon_i \geq 0$ and $\varepsilon_1 + \varepsilon_2 > 0$.
\end{enumerate}
The second condition is equivalent to all geodesics having finite length in $\M$.
\medskip

An important subclass of convex non-trapping manifolds are simple manifolds. We say that a compact Riemannian
manifold $(\M, \g)$ is simple if it satisfies the following properties
\begin{enumerate}
     \item[(a)] the boundary is strictly convex,
     \item[(b)] there are no conjugate points on any geodesic.
\end{enumerate}
A simple $n$-dimensional Riemannian manifold is diffeomorphic to a closed ball in $\R^n$, and any pair of
points in the manifold are joined by an unique geodesic.
\medskip

Given $(x,\theta)\in S\M$, there exist a unique geodesic $\gamma_{x,\theta}$ associated to $(x,\theta)$ which is maxmimally defined on a finite intervall $[\tau_-(x,\theta),\tau_+(x,\theta)]$, with $\gamma_{x,\theta}(\tau_\pm(x,\theta))\in\p\M$. We define the geodesic flow $\Phi_t$ as following
\begin{equation}\label{2.7}
\Phi_t:S\M\to S\M,\quad \Phi_t(x,\theta)=(\gamma_{x,\theta}(t),\dot{\gamma}_{x,\theta}(t)),\quad t\in [\tau_-(x,\theta),\tau_+(x,\theta)],
\end{equation}
and $\Phi_t$ is a flow, that is, $\Phi_t\circ\Phi_s=\Phi_{t+s}$.
\medskip

Now, we introduce the submanifolds of inner and outer vectors of $S\M$
\begin{equation}\label{2.8}
\p_{\pm}S\M =\set{(x,\theta)\in S\M,\, x \in \p \M,\, \pm\seq{\theta,\nu(x)}< 0},
\end{equation}
where $\nu$ is the unit outer normal to the boundary. Note that $\p_+ S\M$ and $\p_-S\M$ are
compact manifolds with the same boundary $S(\p \M)$, and $\p S\M = \p_+ S\M\cup \p_- S\M$. We denote by  $\mathcal{C}^\infty(\p_+ S\M)$ be the space of smooth functions on the manifold $\p_+S\M$. Thus we can define two functions $\tau_\pm:S\M\to\R$ which satisfy
$$
\tau_-(x,\theta)\leq 0,\quad \tau_+(x,\theta)\geq 0,
$$
$$
\tau_+(x,\theta)=-\tau_-(x,-\theta),
$$
$$
\tau_-(x,\theta)=0,\quad (x,\theta)\in\p_+S\M,
$$
$$
\tau_-(\Phi_t(x,\theta))=\tau_-(x,\theta)-t,\quad \tau_+(\Phi_t(x,\theta))=\tau_+(x,\theta)+t.
$$
For $(x,\theta)\in\p_+ S\M$,  we denote by $\gamma_{x,\theta} : [0,\tau_+(x,\theta)] \to \M$ the maximal
geodesic satisfying the initial conditions $\gamma_{x,\theta}(0) = x$ and $\dot{\gamma}_{x,\theta}(0) = \theta$. For each smooth $1$-form $A\in\mathcal{C}^\infty(\M,T^*\M)$, $A=a_jdx^j$ we introduce the smooth symbol function $\sigma_A\in\mathcal{C}^\infty(S\M)$ given by
\begin{equation}\label{2.9}
\sigma_A(x,\theta)=\sum_{j=1}^na_j(x)\theta^j=\seq{A^\sharp(x),\theta},\quad (x,\theta)\in S\M.
\end{equation}

The Riemannian scalar product on $T_x\M$ induces the volume form on $S_x\M$,
denoted by $\dd \omega_x(\theta)$ and given by
$$
\dd \omega_x(\theta)=\sqrt{\abs{\g}} \, \sum_{k=1}^n(-1)^k\theta^k \dd \theta^1\wedge\cdots\wedge \widehat{\dd \theta^k}\wedge\cdots\wedge \dd \theta^n.
$$
As usual, the notation $\, \widehat{\cdot} \,$ means that the corresponding factor has been dropped.
We introduce the volume form $\dvv$ on the manifold $S\M$ by
$$
\dvv (x,\theta)=\dd\omega_x(\theta)\wedge \dv,
$$
where $\dv$ is the Riemannnian volume form on $\M$. By Liouville's theorem, the form $\dvv$ is preserved by the geodesic flow. The
corresponding volume form on the boundary $\p S\M =\set{(x,\theta)\in S\M,\, x\in\p \M}$ is given
by
$$
\dss=\dd\omega_x(\theta) \wedge \ds,
$$
where $\ds$ is the volume form of $\p \M$.
\medskip

We now recall the Santal\'o formula
\begin{equation}\label{2.10}
\int_{S\M} F(x,\theta) \dvv (x,\theta)=\int_{\p_+S\M}\para{\int_0^{\tau_+(x,\theta)} F\para{\Phi_t(x,\theta)}dt}\mu(x,\theta)\dss
\end{equation}
for any $F\in\mathcal{C}(S\M)$.
\medskip

Let $L^2_\mu(\p_+S\M)$ be the space of square integrable functions with respect to the measure $\mu(x,\theta)\dss$ with
$\mu(x,\theta)=\abs{\seq{\theta,\nu(x)}}$. This Hilbert space is endowed with the scalar product
\begin{equation}\label{2.11}
\para{u,v}_\mu=\int_{\p_+S\M}u(x,\theta) \overline{v}(x,\theta) \mu(x,\theta)\dss.
\end{equation}
\subsection{Geodesical ray transform of $1$-forms}
 The ray transform of $1$-forms on a simple Riemannian manifold $(\M,\g)$ is the linear operator:
$$
\I_1:\mathcal{C}^\infty(\M, T^*\M)\To \mathcal{C}^\infty(\p_+S\M)
$$
defined by
$$
\I_1 (A)(x,\theta)=\int_{\gamma_{x,\theta}}A=\sum_{j=1}^n\int_0^{\tau_+(x,\theta)} a_j(\gamma_{x,\theta}(t))\dot{\gamma}^j_{x,\theta}(t)dt=\int_0^{\tau_+(x,\theta)}\sigma_A(\Phi_t(x,\theta))dt,
$$
where $\gamma_{x,\theta}: [0,\tau_+(x,\theta)]\to\M$ is a maximal geodesic satisfying the initial conditions $\gamma_{x,\theta}(0)=x$ and $\dot{\gamma}_{x,\theta}(0)=\theta$.
It is easy to see that $\I_1(d\varphi)=0$ for any smooth function $\varphi$ in $\M$ with $\varphi_{|\p\M}=0$. It is known that $\I_1$ is injective on the space of solenoidal $1$-forms satisfying $\delta A=0$ for simple metric $\g$. In other words, $A\in H^1(\M,T^*\M)$ and $\I_1(A)=0$ implies $A^s=0$, i.e., $A=d\varphi$ with some $\varphi$ vanishing on $\p\M$. So we have
\begin{equation}\label{2.111}
\abs{\I_1(A)(x,\theta)}=\abs{\I_1(A^s)(x,\theta)}\leq C\norm{A^s}_{\mathcal{C}^0}, \quad A\in \mathcal{C}^0(\M,T^*\M).
\end{equation}
\medskip

 We will now, determine the adjoint $\I_1^*$ of $\I_1$. The ray transform $\I_1$ is a bounded operator from $L^2(\M,T^*\M)$ into
$L^2_\mu(\p_+S\M)$. For $A\in L^2(\M,T^*\M)$ and $\Psi\in L_\mu^2(\p_+S\M)$, we get
\begin{eqnarray}\label{2.12}
\para{\I_1(A),\Psi}_\mu &=&\int_{\p_+S\M}\I(A)(x,\theta)\overline{\Psi}(x,\theta)\mu(x,\theta)\dss\cr
&=&\int_{\p_+S\M}\para{\int_0^{\tau_+(x,\theta)} \sigma_A\para{\Phi_t(x,\theta)}dt}\overline{\Psi}(x,\theta)\mu(x,\theta)\dss\cr
&=& \int_{S\M} \sigma_A(x,\theta) \overline{\widecheck{\Psi}}(x,\theta)\dvv (x,\theta)=\para{A,\I_1^*(\Psi)},
\end{eqnarray}
where the adjoint $\I_1^*:L^2_\mu(\p_+S\M)\To L^2(\M,T^*\M)$ is given by
\begin{equation}\label{2.13}
\para{\I_1^*\Psi(x)}_j=\int_{S_x\M}\theta^j\widecheck{\Psi}(x,\theta)\, \dd\omega_x(\theta)
\end{equation}
where $\widecheck{\Psi}$ is the extension of the function $\Psi$ from $\p_+S\M$ to $S\M$ constant on every orbit of the geodesic flow, i.e.
$$
\widecheck{\Psi}(x,\theta)=\Psi\big(\gamma_{x,\theta}(\tau_-(x,\theta)),\dot{\gamma}_{x,\theta}(\tau_{-}(x,\theta))\big)=\Psi(\Phi_{\tau_-(x,\theta)}(x,\theta)),\quad (x,\theta)\in S\M.
$$
The  ray transform of $1$-forms on a simple Riemannian manifold can be extend to the bounded operator
$$
\I_1:H^k(\M,T^*\M)\To H^k (\p_+S\M).
$$
Now, we recall some properties of the  ray transform of $1$-forms on a simple Riemannian manifold proved in \cite{[SU3]}. Let $(\M,\g)$ be a simple metric, we assume that $\g$ extends smoothly as a simple metric on $\M_1^{\textrm{int}}\Supset \M$ and let $N_1=\I_1^*\I_1$. Then there exist
$C_1>0, C_2>0$ such that
\begin{equation}\label{2.14}
C_1\norm{A^s}_{L^2(\M)}\leq\norm{N_1(A)}_{H^1(\M_1)}\leq C_2\norm{A^s}_{L^2(\M)}
\end{equation}
for any $A\in L^2(\M,T^*\M)$. If $\mathcal{O}$ is an open set of the simple Riemannian manifold $(\M_{1},\g)$, the normal operator $N_1=\I_1^*\I_1$ is an
elliptic  pseudodifferential operator of order $-1$ on $\mathcal{O}$ (see Appendix B for more details) whose principal symbol is $\varrho(x,\xi)=(\varrho_{jk}(x,\xi))_{1\leq j,k\leq n}$, where
$$
\varrho_{j,k}(x,\xi)=\frac{c_n}{\abs{\xi}}\para{\g_{jk}-\frac{\xi_j\xi_k}{\abs{\xi}^2}}.
$$
Therefore for each $k\geq 0$ there exists a constant $C_k>0$ such that for all $A\in H^k(\M,T^*\M)$ compactly supported in $\mathcal{O}$
\begin{equation}\label{2.15}
\norm{N_1(A)}_{H^{k+1}(\M_{1})}\leq C_k\norm{A^s}_{H^k(\mathcal{O})}.
\end{equation}
\subsection{Geodesical ray transform of function}
The ray transform (also called geodesic X-ray transform) on a convex non trapping manifold $\M$ is the linear operator
\begin{equation}\label{2.16}
\I_0:\mathcal{C}^\infty(\M)\To \mathcal{C}^\infty(\p_+S\M)
\end{equation}
defined by the equality
\begin{equation}\label{2.17}
\I_0 f(x,\theta)=\int_0^{\tau_+(x,\theta)}f(\gamma_{x,\theta}(t))\, \dd t.
\end{equation}
The right-hand side of (\ref{2.17}) is a
smooth function on $\p_+S\M$ because the integration limit $\tau_+(x,\theta)$ is a smooth function on $\p_+S\M$, see Lemma 4.1.1 of \cite{[Sh]}.
The ray transform on a convex non trapping manifold $\M$ can be extended as a bounded operator
\begin{equation}\label{2.18}
\I_0:H^k(\M)\To H^k(\p_+S\M)
\end{equation}
for every integer $k\geq 1$, see Theorem 4.2.1 of \cite{[Sh]}.
\medskip

The ray transform $\I_0$ is a bounded operator from $L^2(\M)$ into
$L^2_\mu(\p_+S\M)$. The adjoint $\I_0^*:L^2_\mu(\p_+S\M)\to L^2(\M)$ is given by
\begin{equation}\label{2.19}
\I_0^*\Psi(x)=\int_{S_x\M}\widecheck{\Psi}(x,\theta)\, \dd\omega_x(\theta)
\end{equation}
where $\widecheck{\Psi}$ is the extension of the function $\Psi$ from $\p_+S\M$ to $S\M$ constant on every orbit of the geodesic flow, i.e.
$$
\widecheck{\Psi}(x,\theta)=\Psi(\gamma_{x,\theta}(\tau_+(x,\theta))).
$$
Let $(\M,\g)$ be a simple metric, we assume that $\g$ extends smoothly as a simple metric on $\M_1^{\textrm{int}}\Supset \M$ and let $N_0=\I_0^*\I_0$. Then there exist
$C_1>0, C_2>0$ such that
\begin{equation}\label{2.20}
C_1\norm{f}_{L^2(\M)}\leq\norm{N_0(f)}_{H^1(\M_1)}\leq C_2\norm{f}_{L^2(\M)}
\end{equation}
for any $f\in L^2(\M)$. If $\mathcal{O}$ is an open set of the simple Riemannian manifold $(\M_{1},\g)$, the normal operator $N_0$ is an
elliptic  pseudodifferential operator of order $-1$ on $\Omega$ whose principal symbol is a multiple of $\abs{\xi}^{-1}$ (see \cite{[SU2]}).
Therefore there exists a constant $C_k>0$ such that for all $f\in H^k(\mathcal{O})$ compactly supported in $\mathcal{O}$
\begin{equation}\label{2.21}
\norm{N_0(f)}_{H^{k+1}(\M_{1})}\leq C_k\norm{f}_{H^k(\mathcal{O})}.
\end{equation}

\section{Geometrical optics solutions of the magnetic Schr\"odinger equation}
\setcounter{equation}{0}
We now proceed to the construction of geometrical optics solutions to the  magnetic Schr\"odinger equation. We extend the  manifold $(\M,\g)$ into a simple
manifold $\M_1^{\textrm{int}} \Supset \M$. The potentials $A_{1},A_{2}$ may also be extended to $\M_{1}$ and their $H^1(\M_{1},T^*\M_1)$ norms may be bounded by
$M_{0}$. Since $A_{1}=A_{2}$ and $V_1=V_2$ near the boundary, their extension outside $\M$ can be taken the same so that $A_{1}=A_{2}$ and $V_1=V_2$ in
$\M_{1} \setminus \M$.
\medskip

Our construction here is a modification of a similar result in \cite{[BellDSSF]}, which dealt with the situation of the Schr\"odinger equation without magnetic potential.
\medskip

We suppose, for a moment, that we are able to find a function $\psi\in\mathcal{C}^2(\M)$ which satisfies the eikonal equation
\begin{equation}\label{3.1}
\abs{\nabla\psi}^2=\sum_{i,j=1}^n\g^{ij}\frac{\p\psi}{\p x_i}\frac{\p\psi}{\p
x_j}=1,\qquad \forall x\in \M,
\end{equation}
and assume that there exist a function $\alpha\in H^1(\R,H^2(\M))$ which solves the transport equation
\begin{equation}\label{3.2}
\p_t \alpha+\seq{d\psi,d\alpha}+\frac{1}{2} (\Delta \psi)\alpha=0,\qquad \forall t\in\R,\, x\in\M,
\end{equation}
which satisfies for some $T_{0}>0$
\begin{equation}\label{3.3}
\alpha(t,x)|_{t\leq 0}=\alpha(t,x)|_{t\geq T_{0}}=0,\quad \forall x\in \M.
\end{equation}
moreover, we assume that there exist a function  $\beta\in H^1(\R,H^2(\M)) $ which solves the transport equation
\begin{equation}\label{3.4}
\p_t \beta+\seq{d\psi,d\beta}-i \seq{A,d\psi}\beta=0,\qquad \forall t\in\R,\, x\in\M.
\end{equation}
We also introduce the norm $\norm{\cdot}_*$ given by
\begin{equation}\label{3.5}
\norm{\alpha}_*=\norm{\alpha}_{H^1(0,T_0;H^2(\M))}.
\end{equation}
\begin{lemma}\label{L3.1}
Let $A\in \mathcal{C}^1(\M,T^*\M)$ and $V\in W^{1,\infty}(\M)$. The magnetic Schr\"odinger equation
\begin{align*}
(i\p_t+\h_{A,V})u&=0,\quad \textrm{in}\quad Q,\\
u(0,x)&=0,\quad \textrm{in}\quad \M,
\end{align*}
has a solution of the form
\begin{equation}\label{3.6}
u(t,x)=\alpha(2\lambda t,x)\beta(2\lambda t,x)e^{i\lambda(\psi(x)-\lambda
t)}+v_\lambda(t,x),
\end{equation}
such that
\begin{equation}\label{3.7}
u\in \mathcal{C}^1(0,T;L^2(\M))\cap\mathcal{C}(0,T;H^2(\M)),
\end{equation}
where $v_\lambda(t,x)$ satisfies
\begin{align*}
v_\lambda(t,x)&=0,\quad (t,x)\in \Sigma , \\
v_\lambda(0,x)&=0,\quad x\in \M.
\end{align*}
Furthermore, there exist $C>0$ such that, for all $\lambda \geq T_0/2T$ the following estimates hold true.
\begin{equation}\label{3.8}
\norm{v_\lambda(t,\cdot)}_{H^k(\M)}\leq C\lambda^{k-1}\norm{\alpha}_*,\qquad
k=0,1.
\end{equation}
The constant $C$ depends only on $T$ and $\M$ (that is $C$ does
not depend on $a$ and $\lambda$). The result remains true if the initial condition $u(0,x)=0$ is replaced
by the final condition $u(T,x)=0$ provided $\lambda \geq T_{0}/2T$; in this case $v_{\lambda}$ is such that $v_{\lambda}(T,x)=0$.
\end{lemma}
\begin{proof}
Let us consider
\begin{equation}\label{3.9}
R(t,x)=-\para{i\partial_t+\h_{A,V}}\para{(\alpha\beta)(2\lambda
t,x)e^{i\lambda(\psi-\lambda t)}}.
\end{equation}
Let $v$ solve the following homogenous boundary value problem
\begin{equation}\label{3.10}
\left\{
\begin{array}{llll}
\para{i\partial_t+\h_{A,V}}v(t,x)=R(t,x)
& \textrm{in }\,\, Q,\cr
v(0,x)=0,& \textrm{in
}\,\,\M,\cr
v(t,x)=0 & \textrm{on} \,\, \Sigma.
\end{array}
\right.
\end{equation}
To prove our Lemma it would be enough to show that $v$ satisfies the estimates (\ref{3.8}).
The case where the condition $u(T,x)=0$ is imposed rather than the initial condition may be handled in a similar
fashion by imposing the corresponding condition $v(T,x)=0$ on $v$ since $\alpha(2\lambda T,\cdot)=0$ if $\lambda>T_{0}/2T$.
By a simple computation, we have
\begin{align}\label{3.11}
-R(t,x)&=e^{i\lambda(\psi(x)-\lambda t)}\h_{A,V}\para{(\alpha\beta)(2\lambda t,x)}\cr
&\quad +2i\lambda e^{i\lambda(\psi(x)-\lambda t)}\beta(2\lambda t,x)\para{\p_t\alpha+\seq{d\psi,d\alpha}+\frac{\alpha}{2}\Delta\psi}(2\lambda t,x)\cr
&\quad +2i\lambda e^{i\lambda(\psi(x)-\lambda t)}\alpha(2\lambda t,x)\para{\p_t\beta+\seq{d\psi,d\beta}-i \seq{A,d\psi}\beta}(2\lambda t,x)\cr
&\quad+\lambda^2 \alpha(2\lambda t,x) e^{i\lambda(\psi(x)-\lambda t)}\para{1-\abs{d\psi}^2}.
\end{align}
Taking into account  (\ref{3.1})-(\ref{3.2}) and (\ref{3.4}), the right-hand side of (\ref{3.11}) becomes
\begin{align}\label{3.12}
R(t,x)&=-e^{i\lambda(\psi(x)-\lambda t)}\h_{A,V}\para{(\alpha\beta)(2\lambda t,x)}
\cr &\equiv-e^{i\lambda(\psi(x)-\lambda t)}R_0(2\lambda t,x).
\end{align}
Since $R_0\in H^1_0(0,T;L^2(\M))$ for $\lambda>T_0/2T$, by Lemma \ref{L7.2}, we find
\begin{equation}\label{3.13}
v_\lambda\in \mathcal{C}^1(0,T;L^2(\M))\cap\mathcal{C}(0,T;H^2(\M)\cap
H^1_0(\M)).
\end{equation}
Furthermore, there is a constant $C>0$, such that
\begin{align}\label{3.14}
\norm{v_\lambda(t,\cdot)}_{L^2(\M)}&\leq C\int_0^T\norm{R_0(2\lambda
s,\cdot)}_{L^2(\M)}\, \dd s \displaybreak[1] \\ \nonumber &\leq
\frac{C}{\lambda}\int_\R\norm{R_0(s,\cdot)}_{L^2(\M)} \, \dd s \displaybreak[1] \\ \nonumber  &\leq
\frac{C}{\lambda}\norm{\alpha}_*.
\end{align}
Moreover, for any $\eta>0$, we have
\begin{multline}\label{3.15}
\norm{\nabla v_\lambda (t,\cdot)}_{L^2(\M)}
\leq
C\eta\int_0^T\para{\lambda^2\norm{R_0(2\lambda
s,\cdot)}_{L^2(\M)}+\lambda \norm{\p_t R_0(2\lambda
s,\cdot)}_{L^2(\M)}}\, \dd s\\ +\eta^{-1}\int_0^T\norm{
R_0(2\lambda s,\cdot)}_{L^2(\M)}\, \dd s.
\end{multline}
Finally, choosing $\eta=\lambda^{-1}$, we obtain
\begin{align}\label{3.16}
\norm{\nabla v_\lambda(t,\cdot)}_{L^2(\M)} &\leq
C\para{\int_\R\norm{R_0(s ,\cdot)}_{L^2(\M)} \, \dd s+\int_\R\norm{\p_t R_0(s ,\cdot)}_{L^2(\M)} \, \dd s} \nonumber \\ &\leq  C\norm{\alpha}_*.
\end{align}
Combining (\ref{3.16}) and (\ref{3.14}), we immediately deduce the estimate (\ref{3.8}).
\end{proof}
\medskip

We will now construct the phase function $\psi$ solution to the eikonal equation (\ref{3.1}) and the amplitudes $\alpha$ and $\beta$ solutions to the
transport equations (\ref{3.2})-(\ref{3.4}).
\medskip

Let $y\in \p \M_1$. Denote points in $\M_1$ by $(r,\theta)$
where $(r,\theta)$ are polar normal coordinates in $\M_1$ with center
$y$. That is $x=\exp_{y}(r\theta)$ where $r>0$ and
$$
\theta\in S_{y}\M_1=\set{\theta\in T_{y}\M_1,\,\,\abs{\theta}=1}.
$$
In these coordinates (which depend on the choice of $y$) the
metric takes the form
$$
\widetilde{\g}(r,\theta)=\dd r^2+\g_0(r,\theta),
$$
where $\g_0(r,\theta)$ is a smooth positive definite metric.
For any function $u$ compactly supported in $\M$, we set for $r>0$ and $\theta\in S_y\M_1$
$$
\widetilde{u}(r,\theta)=u(\exp_{y}(r\theta)),
$$
where we have extended $u$ by $0$ outside $\M$.
An explicit solution to the eikonal equation (\ref{3.1}) is the geodesic distance function to $y \in \p \M_1$
\begin{equation}\label{3.17}
\psi(x)=d_\g(x,y).
\end{equation}
By the simplicity assumption, since $y\in \M_1\backslash\overline{\M}$, we have $\psi\in\mathcal{C}^\infty(\M)$ and
\begin{equation}\label{3.18}
\widetilde{\psi}(r,\theta)=r=d_\g(x,y).
\end{equation}
The next step is to solve the transport equation (\ref{3.2}). Recall that
if $f(r)$ is any function of the geodesic distance $r$, then
\begin{equation}\label{3.19}
\Delta_{\widetilde{\g}}f(r)=f''(r)+\frac{\rho^{-1}}{2}\frac{\p
\rho}{\p r}f'(r).
\end{equation}
Here $\rho=\rho(r,\theta)$ denotes the square of the volume element in geodesic polar coordinates.
The transport equation (\ref{3.2}) becomes
\begin{equation}\label{3.20}
\frac{\p \widetilde{\alpha}}{\p t}+\frac{\p \widetilde{\psi}}{\p
r}\frac{\p \widetilde{\alpha}}{\p
r}+\frac{1}{4}\widetilde{\alpha}\rho^{-1}\frac{\p \rho}{\p r}\frac{\p
\widetilde{\psi}}{\p r}=0.
\end{equation}
Thus $\widetilde{\alpha}$ satisfies
\begin{equation}\label{3.21}
\frac{\p \widetilde{\alpha}}{\p t}+\frac{\p \widetilde{\alpha}}{\p
r}+\frac{1}{4}\widetilde{\alpha}\rho^{-1}\frac{\p \rho}{\p r}=0.
\end{equation}
Let $\phi\in\mathcal{C}_0^\infty(\R)$ and $\Psi\in H^2(\p_+S\M)$. Let us write $\widetilde{\alpha}$ in the form
\begin{equation}\label{3.22}
\widetilde{\alpha}(t,r,\theta)=\rho^{-1/4}\phi(t-r)\Psi(y,\theta).
\end{equation}
Direct computations yield
\begin{equation}\label{3.23}
\frac{\p \widetilde{\alpha}}{\p
t}(t,r,\theta)=\rho^{-1/4}\phi'(t-r)\Psi(y,\theta).
\end{equation}
and
\begin{equation}\label{3.24}
\frac{\p \widetilde{\alpha}}{\p
r}(t,r,\theta)=-\frac{1}{4}\rho^{-5/4}\frac{\p\rho}{\p
r}\phi(t-r)\Psi(y,\theta)-\rho^{-1/4}\phi'(t-r)\Psi(y,\theta).
\end{equation}
Finally, (\ref{3.24}) and (\ref{3.23}) yield
\begin{equation}\label{3.25}
\frac{\p \widetilde{\alpha}}{\p t}(t,r,\theta)+\frac{\p \widetilde{\alpha}}{\p
r}(t,r,\theta)=-\frac{1}{4}\rho^{-1}\widetilde{\alpha}(t,r,\theta)\frac{\p\rho}{\p
r}.
\end{equation}
Now if we assume that $\mathrm{supp}(\phi) \subset (0,1)$,  then for any $x=\exp_y(r\theta)\in \M$, it is easy to
see that $\widetilde{\alpha}(t,r,\theta)=0$ if $t\leq 0$ and $t\geq T_0$ for some $T_0>1+\mathop{\rm diam} \M_{1}$.\\
In  geodesic polar coordinates the gradient vector $\nabla\psi(x)$ is given by $\dot{\gamma}_{y,\theta}(r)$ we give the proof in Appendix C (see also \cite{[ITOH-SAKAI]}), then 
$$
\seq{\widetilde{A}(r,y,\theta),d\psi}=\seq{\widetilde{A}^\sharp(r,y,\theta),\nabla\psi}=\widetilde{\sigma}_A(\Phi_r(y,\theta)).
$$
The transport equation (\ref{3.4}) becomes
\begin{equation}\label{3.26}
\frac{\p \widetilde{\beta}}{\p t}+\frac{\p \widetilde{\psi}}{\p
r}\frac{\p \widetilde{\beta}}{\p
r} -i\widetilde{\sigma}_A(r,y,\theta)\widetilde{\beta}=0.
\end{equation}
where $\widetilde{\sigma}_A(r,y,\theta):=\sigma_A(\Phi_r(y,\theta))=\seq{\dot{\gamma}_{y,\theta}(r),A^\sharp(\gamma_{y,\theta}(r))}$.
Thus $\widetilde{\beta}$ satisfies
\begin{equation}\label{3.27}
\frac{\p \widetilde{\beta}}{\p t}+\frac{\p \widetilde{\beta}}{\p
r}-i\widetilde{\sigma}_A(r,y,\theta)\widetilde{\beta}=0.
\end{equation}
Thus, we can choose $\widetilde{\beta}$ as following
$$
\widetilde{\beta}(t,y,r,\theta)=\exp\para{i\int_0^t\widetilde{\sigma}_A(r-s,y,\theta)ds}.
$$
Hence (\ref{3.4}) is solved.
\section{Stable determination of the solenoidal part of the magnetic field}
\setcounter{equation}{0}
In this section, we prove the stability estimate of the  solenoidal part $A^s$ of the magnetic field $A$. We are going to use the geometrical
optics solutions constructed in the previous section; this will provide information on the geodesic ray transform of the difference of magnetic potentials.
\subsection{Preliminary estimates}
The main purpose of this section is to present a preliminary estimate, which relates the
difference of the potentials to the Dirichlet-to-Neumann map.
As before, we let $A_1,\,A_2\in\mathscr{A}(m_1,k)$ and $V_1,V_2\in\mathscr{V}(m_2)$ such that $A_1=A_2$, $V_1=V_2$ near  the boundary $\p\M$. We set
$$
A(x)=(A_1-A_2)(x),\quad V(x)=(V_1-V_2)(x).
$$
Recall that we have extended $A_{1},A_{2}$ as $H^1(\M_{1},T^*\M_1)$ in such a way that $A=0$ and $V=0$ on $\M_{1} \setminus \M$.
\begin{lemma}\label{L4.1}
Let $T>0$. There exist $C>0$ such that for any $\alpha_j,\beta_j\in H^1(\R, H^2(\M))$
satisfying the transport equation (\ref{3.2}) with (\ref{3.3}), the following estimate holds true:
\begin{multline}\label{4.1}
\abs{2\lambda\int_{0}^T\!\!\!\!\int_{\M}\seq{A,d\psi}(\alpha_2\overline{\alpha}_1)(2\lambda t,x)(\beta_2\overline{\beta}_1)(2\lambda,x)\,\dv \, \dd t } \cr
\leq C\para{\lambda^{-1}+\lambda^2\norm{\Lambda_{A_1,V_1}-\Lambda_{A_2,V_2}}}\norm{\alpha_1}_*\norm{\alpha_2}_*
\end{multline}
for all $\lambda > T_{0}/2T$.
\end{lemma}
\begin{proof} First, if $\alpha_2$ satisfies (\ref{3.2}), $\beta_2$ satisfie (\ref{3.3}), and $\lambda>T_0/2T$, Lemma \ref{L3.1} guarantees
the existence of a geometrical optics solution $u_2$
\begin{equation}\label{4.2}
u_2(t,x)=(\alpha_2\beta_2)(2\lambda t,x)e^{i\lambda(\psi(x)-\lambda
t)}+v_{2,\lambda}(t,x),
\end{equation}
to the Schr\"odinger equation corresponding to the potentials $A_2$ and $V_2$,
$$
\para{i\p_t+\h_{A_2,V_2}}u(t,x)=0\quad \textrm{in}\,Q, \quad u(0,\cdot)=0 \quad \textrm{in}\, \M,
$$
where $v_{2,\lambda}$ satisfies
\begin{gather}\label{4.3}
\lambda\norm{v_{2,\lambda}(t,\cdot)}_{L^2(\M)}+\norm{\nabla
v_{2,\lambda}(t,\cdot)}_{L^2(\M)}\leq C\norm{\alpha_2}_*,
\\ \nonumber
v_{2,\lambda}(t,x)=0,\quad\forall (t,x)\in\,\Sigma.
\end{gather}
Moreover
$$
u_2\in  \mathcal{C}^1(0,T;L^2(\M))\cap \mathcal{C}(0,T;H^2(\M)).
$$
Let us denote by $f_\lambda$ the function
$$
f_\lambda(t,x)=(\alpha_2\beta_2)(2\lambda t,x)e^{i\lambda(\psi(x)-\lambda
t)},\quad  (t,x)\in\Sigma.
$$
Let us consider $v$ the solution of the following non-homogenous boundary value problem
\begin{equation}\label{4.4}
\left\{\begin{array}{lll}
\para{i\p_t+\h_{A_1,V_1}} v=0, & (t,x)\in Q,\cr
v(0,x)=0, & x\in \M,\cr
v(t,x)=u_2(t,x):=f_{\lambda}(t,x), & (t,x)\in \Sigma.
\end{array}
\right.
\end{equation}
Denote $ w=v-u_2$. Therefore, $w$ solves the following homogenous boundary value problem for the magnetic Schr\"odinger equation
$$
\left\{\begin{array}{lll}
\para{i\p_t+\h_{A_1,V_1}(x)}w(t,x)=2i\seq{A, du_2}+W(x)u_2(t,x) & (t,x)\in Q,\cr
w(0,x)=0, & x\in \M,\cr
w(t,x)=0, & (t,x)\in \Sigma,
\end{array}
\right.
$$
where
$$
W(x)=i\delta(A)-\abs{A_2}^2+\abs{A_1}^2+V\equiv W_A+V.
$$
Using the fact that  $W(x)u_2\in W^{1,1}(0,T;L^2(\M))$ with $u_2(0,\cdot)\equiv 0$, by Lemma \ref{L7.2}, we deduce that
$$
w\in \mathcal{C}^1(0,T;L^2(\M))\cap \mathcal{C}(0,T;H^2(\M)\cap
H^1_0(\M)).
$$
Therefore, we have constructed a special solution
  $$ u_1\in \mathcal{C}^1(0,T;L^2(\M))\cap \mathcal{C}(0,T;H^2(\M)) $$
to the backward magnetic Schr\"odinger equation
\begin{align*}
\para{i\partial_t+\h_{A_1,V_1}}u_1(t,x)&=0,  \quad (t,x) \in Q, \\
u_1(T,x)&=0,  \quad x \in \M,
\end{align*}
having the special form
\begin{equation}\label{4.5}
u_1(t,x)=(\alpha_1\beta_1)(2\lambda t,x)e^{i\lambda(\psi(x)-\lambda
t)}+v_{1,\lambda}(t,x),
\end{equation}
which corresponds to the potentials $A_1$ and $V_1$, where
$v_{1,\lambda}$ satisfies for $\lambda>T_0/2T$
\begin{equation}\label{4.6}
\lambda\norm{v_{1,\lambda}(t,\cdot)}_{L^2(\M)}+\norm{\nabla
v_{1,\lambda}(t,\cdot)}_{L^2(\M)}\leq C\norm{\alpha_1}_*.
\end{equation}
Integrating by parts and using Green's formula (\ref{2.5}), we find
\begin{align}\label{4.7}
\int_0^T\!\!\!\int_\M\para{i\p_t+\h_{A_1,V_1}}w\overline{u}_1\dv \, \dd t
&= \int_0^T\!\!\!\int_\M 2i\seq{A,d u_2}\overline{u}_1\dv \, \dd t+\int_0^T\!\!\!\int_\M (W_{A}+V)(x)u_2\overline{u}_1\dv \, \dd t\cr
&=-\int_0^T\!\!\!\int_{\p \M}(\p_\nu +iA_1\cdot\nu)w\overline{u}_1\ds \, \dd t.
\end{align}
Taking (\ref{4.7}), (\ref{4.4}) into account, we deduce
\begin{multline}\label{4.8}
-\int_0^T\!\!\!\int_\M 2i\seq{A,d u_2}\overline{u}_1(t,x)\dv\,\dd t 
=\int_0^T\!\!\!\int_{\p \M}\para{\Lambda_{A_1,V_1}-\Lambda_{A_2,V_2}} f_{\lambda}(t,x)\overline{h}_\lambda(t,x) \ds \, \dd t\cr
+\int_0^T\!\!\!\int_\M (W_{A}+V)(x)u_2\overline{u}_1\dv \, \dd t
\end{multline}
where $h_\lambda$ is given by
$$
h_\lambda(t,x)=(\alpha_1\beta_1)(2\lambda t,x)e^{i\lambda(\psi(x)-\lambda
t)},\quad (t,x)\in \Sigma.
$$
It follows from (\ref{4.8}), (\ref{4.5}) and (\ref{4.2}) that
\begin{multline}\label{4.9}
2\lambda\int_0^T\!\!\!\int_\M \seq{A,d\psi} (\alpha_2 \overline{\alpha}_1)(2\lambda t,x)(\beta_2\overline{\beta_1})(2\lambda t,x)\dv\,\dd t =\cr
\int_0^T\!\!\!\int_{\p \M} \overline{h}_\lambda \para{\Lambda_{A_1,V_1}-\Lambda_{A_2,V_2}}f_{\lambda} \ds \, \dd t
-2\lambda\int_0^T\!\!\!\int_\M \seq{A,d\psi} (\alpha_2\beta_2)(2\lambda t,x))\overline{v}_{1,\lambda}e^{i\lambda(\psi-\lambda t)}\dv\dd t\cr
+2i\int_0^T\!\!\!\int_\M \seq{A,d(\alpha_2\beta_2)}(2\lambda t,x)\overline{\alpha}_1\overline{\beta}_1(2\lambda t,x)\dv\dd t
+2i\int_0^T\!\!\!\int_\M \seq{A,d(\alpha_2\beta_2)}(2\lambda t,x)\overline{v}_{1,\lambda}(t,x)e^{i\lambda(\psi-\lambda t)}\dv\dd t\cr
+2i\int_0^T\!\!\!\int_\M \seq{A,d v_{2,\lambda}}(\overline{\alpha}_1\overline{\beta}_1)(2\lambda t,x) e^{-i\lambda(\psi-\lambda t)}\dv\dd t
+2i\int_0^T\!\!\!\int_\M \seq{A,d v_{2,\lambda}(t,x)}\overline{v}_{1,\lambda}(t,x)\dv\dd t\cr
+\int_0^T\!\!\!\int_\M (W_{A}+V)(x)u_2(t,x)\overline{u}_1(t,x)\dv dt \cr
= \int_0^T\!\!\!\int_{\p \M} \overline{h}_\lambda \para{\Lambda_{A_1,V_1}-\Lambda_{A_2,V_2}}f_{\lambda} \ds \, \dd t+\mathscr{R}_\lambda
\end{multline}
In view of (\ref{4.6}) and (\ref{4.3}), we have
\begin{equation}\label{4.10}
\abs{\mathscr{R}_\lambda}\leq \frac{C}{\lambda}\norm{\alpha_1}_*\norm{\alpha_2}_*.
\end{equation}
On the other hand, by the trace theorem, we find
\begin{eqnarray}\label{4.11}
\bigg|\int_0^T\!\!\!\int_{\p \M}\para{\Lambda_{A_1,V_1}-\Lambda_{A_2,V_2}}(f_{\lambda}) \overline{h}_\lambda \ds \, \dd t \bigg|
&\leq & \norm{\Lambda_{A_1,V_1}-\Lambda_{A_2,V_2}} \norm{f_\lambda}_{H^{2,1}(\Sigma)}\norm{h_\lambda}_{L^2(\Sigma)}\cr
&\leq & C\lambda^2\norm{\alpha_1}_*\norm{\alpha_2}_*\norm{\Lambda_{A_1,V_1}
-\Lambda_{A_2,V_2}}.
\end{eqnarray}
The estimate (\ref{4.1}) follows easily from (\ref{4.9}), (\ref{4.10}) and (\ref{4.11}).
This completes the proof of the Lemma.
\end{proof}
\begin{lemma}\label{L4.2} There exists $C>0$ such that for any $\Psi\in H^2(\p_+S\M_{1})$, the following estimate
\begin{multline}\label{4.12}
\abs{\int_{S_{y}\M_1}\!\int^{\tau_+(y,\theta)}_0\widetilde{\sigma}_A(s,y,\theta)\Psi(y,\theta) \mu(y,\theta) \, \dd s \, \dd \omega_y(\theta)} \leq
C\norm{\Lambda_{A_1,V_1}-\Lambda_{A_2,V_2}}^{1/2} \norm{\Psi(y,\cdot)}_{H^2(S_y^+\M_{1})}\cr
+\norm{A^s}^2_{\mathcal{C}^0} \norm{\Psi(y,\cdot)}_{L^2(S_y^+\M_{1})}
\end{multline}
holds for any $y\in\p\M_1$.
\end{lemma}
We use the notation
    $$ S_y^+\M_{1} = \big\{\theta \in S_{y}\M_{1} : \langle \nu,\theta \rangle<0 \big\}. $$
\begin{proof}
Following (\ref{3.22}), we pick $T_{0}>1+\mathop{\rm diam} \M_{1}$ and take two solutions to (\ref{3.2}) and (\ref{3.3}) of the form
\begin{align*}
\widetilde{\alpha}_1(t,r,\theta)&=\rho^{-1/4}\phi(t-r)\Psi(y,\theta), \\
\widetilde{\alpha}_2(t,r,\theta)&=\rho^{-1/4}\phi(t-r)\mu(y,\theta).
\end{align*}
We recall that $\mu(y,\theta)=\abs{\langle \nu(y),\theta \rangle}$ is the density of the $L^2$ space where the image of the geodesic ray transform lies.
Now we change variable in the left term of (\ref{4.1}), $x=\exp_{y}(r\theta)$, $r>0$ and
$\theta\in S_{y}\M_1$, we have
\begin{multline}\label{4.13}
2\lambda\int_0^T\!\!\int_\M \seq{A,d\psi} (\overline{\alpha}_1\alpha_2)(2\lambda t,x)(\overline{\beta}_1\beta_2)(2\lambda t,x)  \dv \, \dd t \cr
=2\lambda\int_0^T\!\!\int_{S_{y}\M_1}\!\int_0^{\tau_+(y,\theta)}\widetilde{\sigma}_A(r,y,\theta)(\overline{\widetilde{\alpha}}_1\widetilde{\alpha}_2)(2\lambda t,r,\theta)(\overline{\widetilde{\beta}}_1\widetilde{\beta}_2)(2\lambda t,r,\theta)
\rho^{1/2} \, \dd r \, \dd\omega_y(\theta) \, \dd t\cr
=2\lambda\int_0^T\!\!\int_{S_{y}\M_1}\!\int_0^{\tau_+(y,\theta)}\widetilde{\sigma}_A(r,y,\theta)\phi^2(2\lambda t-r)(\overline{\widetilde{\beta}}_1\widetilde{\beta}_2)(2\lambda t,r,\theta)\Psi(y,\theta) \mu \, \dd r \,
\dd\omega_y(\theta)
\, \dd t\cr
=2\lambda\int_0^T\!\!\int_{S_{y}\M_1}\!\int_\R \widetilde{\sigma}_A(2\lambda t-\tau,y,\theta)\phi^2(\tau)(\widetilde{\beta}_1\widetilde{\beta}_2)(2\lambda t,2\lambda t-\tau,\theta)\Psi(y,\theta) \mu \, \dd \tau \,
\dd\omega_y(\theta)
\, \dd t\cr
=2\lambda\int_0^T\!\!\int_{S_{y}\M_1}\!\int_\R \widetilde{\sigma}_A(2\lambda t-\tau,y,\theta)\phi^2(\tau)\exp\para{i\int_0^{2\lambda t}\widetilde{\sigma}_A(s-\tau,y,\theta)ds}
\Psi(y,\theta) \mu(y,\theta) \, \dd \tau \,
\dd\omega_y(\theta)\cr
=\int_\R\phi^2(\tau) \!\!\int_{S_{y}\M_1}\int_0^T \frac{d}{dt}\exp\para{i\int_0^{2\lambda t}\widetilde{\sigma}_A(s-\tau,y,\theta)ds}
\Psi(y,\theta) \mu \, \dd \tau \,
\dd\omega_y(\theta)\cr
=\int_\R\phi^2(\tau) \!\!\int_{S_{y}\M_1}\cro{\exp\para{i\int_0^{2\lambda T}\widetilde{\sigma}_A(s-\tau,y,\theta)ds}-1}
\Psi(y,\theta) \mu \, \dd \tau \,
\dd\omega_y(\theta).
\end{multline}
By the support properties of the function $\phi$, we get that the left-hand side term in the previous inequality reads
\begin{multline*}
  \int_\R\phi^2(\tau) \!\!\int_{S_{y}\M_1}\cro{\exp\para{i\int_0^{2\lambda T}\widetilde{\sigma}_A(s-\tau,y,\theta)ds}-1}
\Psi(y,\theta) \mu(y,\theta) \, \dd \tau \,
\dd\omega_y(\theta)=\cr
\int_{S_{y}\M_1}\cro{\exp\para{i\int_0^{\tau_+(y,\theta)}\widetilde{\sigma}_A(s,y,\theta)ds}-1}
\Psi(y,\theta) \mu(y,\theta) \, 
\dd\omega_y(\theta).
\end{multline*}
Then, by (\ref{4.13}) and (\ref{4.1}) we get
\begin{multline}\label{4.14}
\abs{\int_{S_{y}\M_1}\para{\exp\para{i\I_1(A)(y,\theta)}-1}
\Psi(y,\theta) \mu(y,\theta) \,  
\dd\omega_y(\theta)}\cr
 \leq C\para{\lambda^{-1}+\lambda^2\norm{\Lambda_{A_1,V_1}-\Lambda_{A_2,V_2}}} \norm{\Psi(y,\cdot)}_{H^2(S^+_y\M_{1})}.
\end{multline}
Finally, minimizing in $\lambda$ in the right hand-side of \eqref{4.14} we obtain
\begin{multline*}
\abs{\int_{S_{y}\M_1}\para{\exp\para{i\I_1(A)(y,\theta)}-1}
\Psi(y,\theta) \mu(y,\theta) \,  
\dd\omega_y(\theta)} \leq C\norm{\Lambda_{A_1,V_1}-\Lambda_{A_2,V_2}}^{1/3} \norm{\Psi(y,\cdot)}_{H^2(S^+_y\M_{1})}.
\end{multline*}
Using the fact that
$$
\exp\para{i\I_1(A)(y,\theta)}-1=i\I_1(A)(y,\theta)-(\I_1(A)(y,\theta))^2\int_0^1\exp(it\I_1(A)(y,\theta))(1-t)dt,
$$
we deduce from (\ref{2.111})
\begin{multline*}
\abs{\int_{S_{y}\M_1}\I_1(A)(y,\theta)
\Psi(y,\theta) \mu(y,\theta) \, \dd \sigma \,
\dd\omega_y(\theta)} \leq C\norm{\Lambda_{A_1,V_1}-\Lambda_{A_2,V_2}}^{1/3} \norm{\Psi(y,\cdot)}_{H^2(S^+_y\M_{1})}\cr
+\norm{\Psi(y,\cdot)}_{L^2(S^+_y\M_{1})}\norm{A^s}_{\mathcal{C}^0}^2.
\end{multline*}
This completes the proof of the lemma.
\end{proof}
\subsection{End of the proof of the stability estimate of the magnetic field}
Let us now complete the proof of the stability estimate of the solenoidal part of the magnetic field. Using Lemma \ref{L4.2}, for any $y\in\p \M_1$ and $\Psi\in H^2(\p_+ S\M_{1})$
we have
\begin{multline*}
\abs{\int_{S_{y}\M_1}\I_1(A)(y,\theta)\Psi(y,\theta) \mu(y,\theta) \, \dd \omega_y(\theta)} \leq C\norm{\Lambda_{A_1,V_1}-\Lambda_{A_2,V_2}}^{1/3}\norm{\Psi(y,\cdot)}_{H^2(S_y^+\M_{1})}.\\
+\norm{\Psi(y,\cdot)}_{L^2(S^+_y\M_{1})}\norm{A^s}_{\mathcal{C}^0(\M,T^*\M)}^2.
\end{multline*}
Integrating with respect to $y\in \p \M_1$ we obtain
\begin{multline}\label{4.15}
\abs{\int_{\p_+S\M_1}\I_1(A)(y,\theta)\Psi(y,\theta)\mu(y,\theta)\dss (y,\theta)} \leq C\norm{\Lambda_{A_1,V_1}-\Lambda_{A_2,V_2}}^{1/3}
\norm{\Psi}_{H^2(\p_+S\M_{1})}\cr
+\norm{\Psi}_{L^2(\p_+S\M_{1})}\norm{A^s}_{\mathcal{C}^0(\M,T^*\M)}^2.
\end{multline}
Now we choose
$$
\Psi(y,\theta)=\I_1\para{N_1(A)}(y,\theta).
$$
Taking into account (\ref{2.15}) and (\ref{4.15}), we obtain
$$
\norm{N_1(A)}^2_{L^2(\M_1)}\leq C\norm{\Lambda_{A_1,V_1}-\Lambda_{A_2,V_2}}^{1/3} \norm{A^s}_{H^1}
+\norm{A^s}_{L^2}\norm{A^s}_{\mathcal{C}^0}^2
$$
By interpolation, it follows that for any $a\in(0,1)$ there exists $k>0$ such that
\begin{align}\label{4.16}
\norm{N_1(A)}^2_{H^1}&\leq C \norm{N_1(A)}^{2a}_{L^2}\norm{N_1(A)}^{2(1-a)}_{H^k}\cr
&\leq C \norm{N_1(A)}^{2a}_{L^2}\norm{A^s}^{2(1-a)}_{H^{k-1}}\cr
&\leq C \norm{N_1(A)}_{L^2}^{2a}\cr
&\leq C\norm{\Lambda_{A_1,V_1}-\Lambda_{A_2,V_2}}^{a/3}+
\norm{A^s}_{L^2}^{a}\norm{A^s}_{\mathcal{C}^0}^{2a}.
\end{align}
Moreover, for any $b\in(0,1)$ there exists $k'>0$ such that
\begin{equation}\label{4.17}
\norm{A^s}_{\mathcal{C}^0}\leq C\norm{A^s}_{H^{n/2+\delta}}\leq C\norm{A^s}^b_{L^2}\norm{A^s}_{H^{k'}}^{1-b}\leq C\norm{A^s}^b_{L^2}.
\end{equation}
Using (\ref{2.14}), we deduce that
$$
\norm{A^s}_{L^2}^2\leq C\norm{\Lambda_{A_1,V_1}-\Lambda_{A_2,V_2}}^{a/3}
+C\norm{A^s}^{a(1+2b)}_{L^2} .
$$
Selecting $a,b\in (0,1)$ such that $a(1+2b)>2$, we deduce that
$$
\norm{A^s}_{L^2}^2\leq C\norm{\Lambda_{A_1,V_1}-\Lambda_{A_2,V_2}}^{a/3}
+C\varepsilon^{(a(1+2b)-2)}
\norm{A^s}^{2}_{L^2}.
$$
So, for $\varepsilon$ small, we deduce
\begin{equation}\label{4.18}
\norm{A^s}_{L^2(\M)}^2\leq C\norm{\Lambda_{A_1,V_1}-\Lambda_{A_2,V_2}}^{a/3}.
\end{equation}
Furthermore by (\ref{4.17}) and (\ref{4.18}) we get
\begin{equation}\label{4.19}
\norm{A^s}_{\mathcal{C}^0}\leq C\norm{\Lambda_{A_1,V_1}-\Lambda_{A_2,V_2}}^{\kappa_1},\quad\kappa_1=ab/6.
\end{equation}
This completes the proof of the H\"older stability estimate of the solenoidal part of the magnetic potential.
\section{Stable determination of the electric potential}
\setcounter{equation}{0}
The goal of this section is to prove a stability estimate for the electric potential. The proof of that stability estimate involves using the stability result we alreaady obtained for the magnetic field. Apply the Hodge decomposition to $A=A_1-A_2=A^s+d\varphi$. Define $A_1'=A_1-\frac{1}{2}d\varphi$ and $A_2'=A_2+\frac{1}{2}d\varphi$ so that $A'=A_1'-A_2'=A^s$. First we remplace the magnetic potential $A_j$ by $A_j'$, $j=1,2$. Since the Dirichlet to Neumann map is invariant under gauge transformation we have
$$
\Lambda_{A_j,V_j}=\Lambda_{A_j',V_j},\quad j=1,2.
$$
Define $\alpha_j$, $\beta_j$ and $u_j$ as in section 4 with $A_j$ replaced by $A_j'$, $j=1,2$.

\begin{lemma}\label{L5.1}
Let $T>0$. There exist $C>0$ such that for any $\alpha_j,\beta_j\in H^1(\R, H^2(\M))$
satisfying the transport equation \eqref{3.2} with \eqref{3.4}, the following estimate holds true:
\begin{multline}\label{5.1}
\abs{\int_{0}^T\!\!\!\!\int_{\M}V(x)(\alpha_1\overline{\alpha}_2)(2\lambda t,x)(\beta_1\overline{\beta}_2)(2\lambda,x)\,\dv \, \dd t } \cr
\leq C\para{\lambda^{-2}+\lambda\|A'\|_{\mathcal{C}^0}+\lambda^2\norm{\Lambda_{A_1,V_1}-\Lambda_{A_2,V_2}}}
\norm{\alpha_1}_*\norm{\alpha_2}_*,
\end{multline}
for all $\lambda > T_{0}/2T$.
\end{lemma}
\begin{proof}
We start with identity (\ref{4.8}) except this time we will isolate the electric potential term on the LHS. 
\begin{multline}\label{5.2}
-\int_0^T\!\!\!\int_\M V(x)u_2\overline{u}_1\dv \, \dd t=
\int_0^T\!\!\!\int_{\p \M}\para{\Lambda_{A'_1,V_1}-\Lambda_{A'_2,V_2}} f_{\lambda}(t,x)\overline{h}_\lambda(t,x) \ds \, \dd t\cr
+\int_0^T\!\!\!\int_\M 2i\seq{A',d u_2}\overline{u}_1(t,x)\dv\,\dd t +\int_0^T\!\!\!\int_\M W_{A'}(x)u_2\overline{u}_1\dv \, \dd t
\end{multline}
where $h_\lambda$ is given by
$$
h_\lambda(t,x)=(\alpha_1\beta_1)(2\lambda t,x)e^{i\lambda(\psi(x)-\lambda
t)},\quad (t,x)\in \Sigma.
$$
It follows from (\ref{5.2}), (\ref{4.5}) and (\ref{4.2}) that
\begin{multline}\label{5.3}
\int_0^T\!\!\!\int_\M V(x) (\alpha_2 \overline{\alpha}_1)(2\lambda t,x)(\beta_2\overline{\beta_1})(2\lambda t,x)\dv\,\dd t =\cr
\int_0^T\!\!\!\int_{\p \M} \overline{h}_\lambda \para{\Lambda_{A'_1,V_1}-\Lambda_{A'_2,V_2}}f_{\lambda} \ds \, \dd t
+\int_0^T\!\!\!\int_\M V(x) (\alpha_2\beta_2)(2\lambda t,x))\overline{v}_{1,\lambda}e^{i\lambda(\psi-\lambda t)}\dv\dd t\cr
+\int_0^T\!\!\!\int_\M V(x)v_{2,\lambda}(\overline{\alpha}_1\overline{\beta}_1)(2\lambda t,x) e^{-i\lambda(\psi-\lambda t)}\dv\dd t
+\int_0^T\!\!\!\int_\M V(x)v_{2,\lambda}(t,x)\overline{v}_{1,\lambda}(t,x)\dv\dd t\cr
+\int_0^T\!\!\!\int_\M W_{A'}(x)u_2(t,x)\overline{u}_1(t,x)\dv dt 
+\int_0^T\!\!\!\int_\M \seq{A',du_2}\overline{u}_1(t,x)\dv dt \cr
= \int_0^T\!\!\!\int_{\p \M} \overline{h}_\lambda \para{\Lambda_{A'_1,V_1}-\Lambda_{A'_2,V_2}}f_{\lambda} \ds \, \dd t+\mathscr{R}'_\lambda
\end{multline}
In view of (\ref{4.6}) and (\ref{4.3}), we have
\begin{equation}\label{5.4}
\abs{\mathscr{R}'_\lambda}\leq \para{\frac{1}{\lambda^2}+\lambda\|A'\|_{\mathcal{C}^0}}\norm{\alpha_1}_*\norm{\alpha_2}_*.
\end{equation}
On the other hand, by the trace theorem, we find
\begin{eqnarray}\label{5.5}
\bigg|\int_0^T\!\!\!\int_{\p \M}\para{\Lambda_{A'_1,V_1}-\Lambda_{A'_2,V_2}}(f_{\lambda}) \overline{h}_\lambda \ds \, \dd t \bigg|
&\leq  &\|\Lambda_{A'_1,V_1}-\Lambda_{A'_2,V_2}\| 
\norm{f_\lambda}_{H^{2,1}(\Sigma)}\norm{h_\lambda}_{L^2(\Sigma)}\cr
&\leq & C\lambda^2\norm{\alpha_1}_*\norm{\alpha_2}_*\|\Lambda_{A'_1,V_1}
-\Lambda_{A'_2,V_2}\|.
\end{eqnarray}
The estimate (\ref{5.1}) follows easily from (\ref{5.3}), (\ref{5.4}).
This completes the proof of the Lemma.
\end{proof}
\begin{lemma}\label{L5.2} There exists $C>0$ and $\kappa_2\in (0,1)$ such that for any $b\in H^2(\p_+S\M_{1})$, the following estimate
\begin{multline}\label{5.6}
\abs{\int_{S_{y}\M_1}\!\int^{\tau_+(y,\theta)}_0\widetilde{V}(s,\theta)b(y,\theta) \mu(y,\theta) \, \dd s \, \dd \omega_y(\theta)} \leq
C\|\Lambda_{A'_1,V_1}-\Lambda_{A'_2,V_2}\|^{\kappa_2} \norm{b(y,\cdot)}_{H^2(S_y^+\M_{1})}.
\end{multline}
holds for any $y\in\p\M_1$.
\end{lemma}
\begin{proof}
Following (\ref{3.22}), we pick $T_{0}>1+\mathop{\rm diam} \M_{1}$ and take two solutions to (\ref{3.2}) and (\ref{3.3}) of the form
\begin{align*}
\widetilde{\alpha}_1(t,r,\theta)&=\rho^{-1/4}\phi(t-r)b(y,\theta), \\
\widetilde{\alpha}_2(t,r,\theta)&=\rho^{-1/4}\phi(t-r)\mu(y,\theta).
\end{align*}
Now we change variable in (\ref{5.1}), $x=\exp_{y}(r\theta)$, $r>0$ and
$\theta\in S_{y}\M_1$, we have
\begin{align*}
\int_0^T\!\!\int_\M &V(x) \alpha_1\overline{\alpha}_2(2\lambda t,x)\beta_1\overline{\beta}_2(2\lambda t,x) \dv \, \dd t \cr
&=\int_0^T\!\!\int_{S_{y}\M_1}\!\int_0^{\tau_+(y,\theta)}\widetilde{V}(r,\theta)\widetilde{\alpha}_1\overline{\widetilde{\alpha}}_2(2\lambda t,r,\theta)\widetilde{\beta}_1\overline{\widetilde{\beta}}_1(2\lambda t,r,\theta)
\rho^{1/2} \, \dd r \, \dd\omega_y(\theta) \, \dd t\cr
&=\int_0^T\!\!\int_{S_{y}\M_1}\!\int_0^{\tau_+(y,\theta)}\widetilde{V}(r,\theta)\phi^2(2\lambda t-r)b(y,\theta) \mu(y,\theta) \, \dd r \,
\dd\omega_y(\theta)
\, \dd t\cr
&=\frac{1}{2\lambda}\int_0^{2\lambda T}\!\!\!\int_{S_{y}\M_1}\!\int_0^{\tau_+(y,\theta)}\widetilde{V}(r,\theta)\phi^2( t-r)b(y,\theta)
\mu(y,\theta) \, \dd r \, \dd\omega_y(\theta) \, \dd t.
\end{align*}
By virtue of Lemma \ref{L5.1}, we conclude that
\begin{multline}\label{5.8}
\abs{\int_0^{\infty}\!\!\!\int_{S_{y}\M_1}\!\int_0^{\tau_+(y,\theta)}\widetilde{V}(r,\theta)\phi^2( t-r)b(y,\theta) \mu(y,\theta) \, \dd r
\, \dd\omega_y(\theta) \, \dd t} \\
\leq C\para{\lambda^{-1}+\lambda^3\|\Lambda_{A',V_1}-\Lambda_{A'_2,V_2}\|
+\lambda^2\|A'\|_{\mathcal{C}^0}}\norm{\phi}^2_{H^3(\R)}\norm{b(y,\cdot)}_{H^2(S^+_y\M_{1})}.
\end{multline}
By the support properties of the function $\phi$, we get that the left-hand side term in the previous inequality reads
\begin{multline*}
    \int_0^{\infty}\!\!\!\int_{S_{y}\M_1}\!\int_0^{\tau_+(y,\theta)}\widetilde{V}(r,\theta)\phi^2( t-r)b(y,\theta) \mu(y,\theta) \, \dd r
    \, \dd\omega_y(\theta) \, \dd t \\
    = \bigg(\int_{-\infty}^{\infty} \phi^2(t) \dd t \bigg) \!\!\!\int_{S_{y}\M_1}\!\int_0^{\tau_+(y,\theta)}\widetilde{V}(r,\theta)b(y,\theta)
    \mu(y,\theta) \, \dd r \, \dd\omega_y(\theta).
\end{multline*}
Then taking acount (\ref{4.19}) we obtain
$$
\int_{S_{y}\M_1}\!\int_0^{\tau_+(y,\theta)}\widetilde{V}(r,\theta)b(y,\theta)
    \mu(y,\theta) \, \dd r \, \dd\omega_y(\theta)\leq  C\para{\lambda^{-1}+\lambda^3\|\Lambda_{A',V_1}-\Lambda_{A'_2,V_2}\|^{\kappa_1}
}\norm{\phi}^2_{H^3(\R)}\norm{b(y,\cdot)}_{H^2(S^+_y\M_{1})}.
$$
Finally, minimizing in $\lambda$ in the right hand-side of the last inequality we obtain
$$
\abs{\int_{S_{y}\M_1}\!\int_0^{\tau_+(y,\theta)}\widetilde{V}(s,\theta)b(y,\theta) \mu(y,\theta) \, \dd s \, \dd\omega_y(\theta)}
 \leq C\|\Lambda_{A',V_1}-\Lambda_{A'_2,V_2}\|^{\kappa_2} \norm{b(y,\cdot)}_{H^2(S^+_y\M_{1})}.
$$
This completes the proof of the lemma.
\end{proof}
\subsection{End of the proof of the stability estimate}
Let us now complete the proof of the stability estimate in Theorem \ref{Th2}. Using Lemma \ref{L5.2}, for any $y\in\p \M_1$ and $b\in H^2(\p_+ S\M_{1})$
we have
\begin{equation*}
\abs{\int_{S_{y}\M_1}\I_0(V)(y,\theta)b(y,\theta) \mu(y,\theta) \, \dd \omega_y(\theta)}
 \leq C\norm{\Lambda_{A_1,V_1}-\Lambda_{A_2,V_2}}^{\kappa_2}\norm{b(y,\cdot)}_{H^2(S_y^+\M_{1})}.
\end{equation*}
Integrating with respect to $y\in \p \M_1$ we obtain
\begin{equation}\label{5.9}
\abs{\int_{\p_+S\M_1}\I_0(V)(y,\theta)b(y,\theta)\mu(y,\theta)\dss (y,\theta)}
 \leq C\norm{\Lambda_{A_1,V_1}-\Lambda_{A_2,V_2}}^{\kappa_2}
\norm{b}_{H^2(\p_+S\M_{1})}.
\end{equation}
Now we choose
$$
b(y,\theta)=\I_0\para{N_0(V)}(y,\theta).
$$
Taking into account (\ref{2.18}) and (\ref{2.20}), we obtain
$$
\norm{N_0(V)}^2_{L^2}\leq C\norm{\Lambda_{A_1,V_1}-\Lambda_{A_2,V_2}}^{\kappa_2} \norm{V}_{H^1}.
$$
By interpolation, it follows that
\begin{align}\label{5.10}
\norm{N_0(V)}^2_{H^1}&\leq C \norm{N_0(V)}_{L^2}\norm{N_0(V)}_{H^2}\cr
&\leq C \norm{N_0(V)}_{L^2}\norm{V}_{H^1}\cr
&\leq C \norm{N_0(V)}_{L^2}\cr
&\leq C\norm{\Lambda_{A_1,V_1}-\Lambda_{A_2,V_2}}^{\kappa_2/2}.
\end{align}
Using (\ref{2.21}), we deduce that
$$
\norm{V}_{L^2(\M)}^2\leq C\norm{\Lambda_{A_1,V_1}-\Lambda_{A_2,V_2}}^{\kappa_2/2}.
$$
This completes the proof of Theorem \ref{Th2}.

\appendix

\section{The Cauchy problem for the magnetic Schr\"odinger equation}
\setcounter{equation}{0}
In this section we will establish existence, uniqueness and continuous dependence on the data of solutions to the magnetic Schr\"odinger equation (\ref{1.2}) with non-homogenous Dirichlet boundary condition $f\in H^{2,1}_0(\Sigma)$. We will use the method of transposition, or adjoint
isomorphism of equations, and we shall solve the case of non-homogenous Dirichlet boundary conditions under stronger assumptions on the data
than those in \cite{[Baudouin-Puel]}.
\medskip

Let $v\in\mathcal{C}^1(\M)$ and $N$ be a smooth real vector field. The following identity holds true (see \cite{[Yao]})
\begin{equation}\label{2.4'}
\big\langle\nabla v, \nabla \seq{N,\nabla\overline{v}}\big\rangle = DN(\nabla v, \nabla\overline{v})+
\frac{1}{2}\dive\big(\abs{\nabla v}^2N\big) -\frac{1}{2}\abs{\nabla v}^2\dive N
\end{equation}
where $D$ is the Levi-Civita connection and $DN$ is the bilinear form on $T_x\M\times T_x\M$ given by
$$
DN(X,Y)=\seq{D_XN,Y},\quad X,Y\in T_x\M.
$$
Here $D_XN$ is the covariant derivative of vector field $N$ with respect to $X$.
\medskip

Let us first review the classical well-posedness results for the Schr\"odinger equation with homogenous boundary conditions. After applying the
transposition method, we establish Theorem \ref{Th0}.
\subsection{Homogenous boundary condition}
Let us consider the following initial and homogenous boundary value problem for the Schr\"odinger equation
\begin{equation}\label{7.1}
\left\{
\begin{array}{llll}
\para{i\partial_t+\h_{A,V}}v(t,x)=F(t,x)  & \textrm{in }\,Q,\cr
v(0,x)=0 & \textrm{in }\,\,\M,\cr
v(t,x)=0 & \textrm{on } \,\, \Sigma.
\end{array}
\right.
\end{equation}
Firstly, it is well known that if $F\in L^1(0,T;L^2(\M))$ then (\ref{7.1}) admits an unique weak solution
\begin{equation}\label{7.2}
v\in \mathcal{C}\para{0,T;L^2(\M)}.
\end{equation}
Multiplying the first equation of (\ref{7.1}) by $\overline{v}$ and using Green's formula and  Gronwall's lemma, we obtain the following estimate
\begin{equation}\label{7.5}
\norm{v(t,\cdot)}_{L^2(\M)}\leq
C\norm{F}_{L^1(0,T;L^2(\M))}, \quad\forall t\in(0,T).
\end{equation}
Now assume that $F\in L^1(0,T;H^1_0(\M))$. Using the classical result of existence and uniqueness of weak solutions in Cazenave and Haraux \cite{[Cazenave-Haraux]} (set for abstract evolution equations), we obtain that the system (\ref{7.1}) has a unique solution $v$ such that
\begin{equation}\label{7.8}
v\in \mathcal{C}(0,T;H^1_0(\M)).
\end{equation}
Multiplying the first equation of (\ref{7.1}) by $\Delta_A\overline{v}$ and using Green's formula and  Gronwall's lemma, we get
\begin{equation}\label{7.7}
\norm{v(t,\cdot)}_{H^1_0(\M)}\leq C\norm{F}_{L^1(0,T;H^1_0(\M))}, \quad\forall t\in(0,T).
\end{equation}

\begin{lemma}\label{L7.2}
Let $T>0$. Suppose that $F\in W^{1,1}(0,T;L^2(\M))$ is such that $F(0,\cdot)\equiv0$. Then the unique solution $v$ of \eqref{7.1}
satisfies
\begin{equation}\label{7.12}
v\in \mathcal{C}^1(0,T;L^2(\M))\cap\mathcal{C}(0,T;H^2(\M)\cap H^1_0(\M)).
\end{equation}
Furthermore there is a constant $C>0$ such that for any $0<\eta \leq 1$, we have
\begin{equation}\label{7.13}
\norm{v(t,\cdot)}_{H^1_0(\M)}\leq C\para{\eta\norm{\p_tF}_{L^1(0,T;L^2(\M))}+\eta^{-1}\norm{F}_{L^1(0,T;L^2(\M))}}.
\end{equation}
\end{lemma}
\begin{proof}
If we consider the equation satisfied by $\p_t v$, (\ref{7.2}) provides the following regularity
   $$v\in \mathcal{C}^{1}(0,T;L^2(\M)).$$
Furthermore, since $F(0,\cdot)=0$, by (\ref{7.5}), there is a constant $C>0$ such that the following estimate holds true
\begin{equation}\label{7.14}
\norm{\p_t v(t,\cdot)}_{L^2(\M)}\leq
C\int_0^T\norm{\p_tF(s,\cdot)}_{L^2(\M)}ds, \qquad\forall t\in(0,T).
\end{equation}
Then, by (\ref{7.1}), we see that  $\h_{A,V} v=-i\p_t v+F \in \mathcal{C}(0,T;L^2(\M))$ and therefore $v\in \mathcal{C}(0,T;H^2(\M))$.
This complete the proof of (\ref{7.12}).\\
Next, multiplying the first equation of (\ref{7.1}) by $\overline{v}$ and integrating by parts, we obtain
\begin{align}\label{7.15}
\textrm{Re}\bigg[\int_\M \para{i\p_tv(t)\overline{v}(t)-\abs{\nabla_A v(t)}^2+V(x)\abs{v(t)}^2}\dve\bigg] \\ \nonumber =\textrm{Re}\!\int_\M\!\!\para{\int_0^t\!\p_tF(s,x) \, \dd s}\overline{v}(t,x)\dve.
\end{align}
Then there exists a constant $C>0$ such that the following estimate holds true
\begin{equation}\label{7.16}
\norm{\nabla v(t)}_{L^2(\M)}^2\leq
C\Big(\norm{\p_tv(t)}_{L^2}\norm{v(t)}_{L^2}+\norm{v(t)}^2_{L^2}  +\int_0^T\!\!\!\int_\M\abs{v(t,x)\p_tF(s,x)}\dve \, \dd s\Big).
\end{equation}
Using (\ref{7.14}) and (\ref{7.5}), we get
\begin{equation}\label{7.17}
\norm{\nabla v(t)}_{L^2(\M)}^2 \leq C\cro{\norm{\p_tF}_{L^1(0,T;L^2(\M))}\norm{F}_{L^1(0,T;L^2(\M))}+\norm{F}_{L^1(0,T;L^2(\M))}^2}.
\end{equation}
Thus, we deduce (\ref{7.13}), and this concludes the proof of Lemma \ref{L7.2}.
\end{proof}
\begin{lemma}\label{L7.3}
Let $T>0$,  be given and let $\mathcal{H}=L^1(0,T;H^1_0(\M))$ or $\mathcal{H}=H^1_0(0,T;L^2(\M))$. Then the mapping
$F\mapsto\p_\nu v$ where $v$ is the unique  solution to \eqref{7.1} is linear and continuous from $\mathcal{H}$ to $L^2(\Sigma)$. Furthermore,
there is a constant $C>0$ such that
\begin{equation}\label{7.18}
\norm{\p_\nu v}_{L^2(\Sigma)}\leq C\norm{F}_{\mathcal{H}}.
\end{equation}
\end{lemma}
\begin{proof}
Let $N$ be a $\mathcal{C}^2$ vector field on $\M$ such that
\begin{equation}\label{7.19}
N(x)=\nu(x),\quad x\in\p\M;\qquad \abs{N(x)}\leq 1,\quad x\in\M.
\end{equation}
Multiply both sides of the first equation in (\ref{7.1}) by $\seq{N,\nabla\overline{v}}$ and integrate over $(0,T)\times\M$, this gives
\begin{align}\label{7.20}
\int_0^T\!\!\!\int_\M F(t,x)&\seq{N,\nabla\overline{v}} \dve \, \dd t \cr &= \int_0^T\!\!\!\int_\M i\p_tv\seq{N,\nabla\overline{v}}\dve \, \dd t
+\int_0^T\!\!\!\int_\M\Delta v\seq{N,\nabla\overline{v}}\dve \, \dd t \cr &\quad+\int_0^T\!\!\!\int_\M(-2i\seq{A,\nabla v}-i(\delta A) v+\abs{A}^2v+V(x)v) \seq{N,\nabla\overline{v}}\dve
=I_1+I_2+I_3.
\end{align}
Consider the first term on the right-hand side of (\ref{7.20}); integrating by parts with respect $t$, we get
\begin{align}\label{7.21}
I_1&=i\cro{\int_\M v\seq{N,\nabla\overline{v}}\dve}_0^T-i\int_0^T\!\!\!\int_\M v\seq{N,\nabla\p_t\overline{v}}\dve \, \dd t \cr
&=i\int_\M v(T,x)\seq{N,\nabla\overline{v}(T,x)}\dve
-i\int_0^T\!\!\!\int_\M\seq{N,\nabla(v\p_t\overline{v})}\dve \, \dd t -\overline{I}_1.
\end{align}
Then, by (\ref{2.3}), we obtain
\begin{align*}
2\textrm{Re}\,I_1&=i\int_\M v(T,x)\seq{N,\nabla\overline{v}(T,x)}\dve +i\int_0^T\!\!\!\int_\M\dive N\, v\p_t\overline{v}\dve \, \dd t \cr
&\quad -i\cro{\int_0^T\!\!\!\int_{\p\M} v\p_t\overline{v}\dse \, \dd t}\cr
&=i\int_\M \!v(T,x)\seq{N,\nabla\overline{v}(T,x)}\dve +\int_0^T\!\!\!\int_\M\!\seq{\nabla\overline{v},\nabla\para{\dive N\,v}}\dve \,\dd t \cr
&\quad+\int_0^T\!\!\!\int_\M \overline{F} \, \dive N\, v\dve \, \dd t -\int_0^T\int_\M \overline{q} \, \dive N\,\abs{v}^2\dve \, \dd t  \cr
&\quad-\Bigg[i\int_0^T\!\!\!\int_{\p\M} v\p_t\overline{v}\dse \, \dd t +\int_0^T\!\!\!\int_{\p\M} \p_\nu\overline{v} \, v \, \dive N\,\dse \, \dd t \Bigg].
\end{align*}
The last term vanishes, using (\ref{7.13}) or (\ref{7.7}), we conclude that
\begin{equation}\label{7.23}
\abs{\textrm{Re}\,I_1}\leq C\norm{F}^2_{\mathcal{H}}.
\end{equation}
On the other hand, by Green's theorem, we get
\begin{equation*}
I_2=-\int_0^T\!\!\!\int_\M\seq{\nabla v, \nabla \seq{N,\nabla\overline{v}}}\dve \, \dd t
+\int_0^T\!\!\!\int_{\p\M}\abs{\p_\nu v}^2\dse \, \dd t.
\end{equation*}
Thus by (\ref{2.4'}), we deduce
\begin{align*}
I_2 &=\int_0^T\!\!\!\int_{\p\M}\abs{\p_\nu v}^2\dse \, \dd t-\frac{1}{2}\int_0^T\!\!\!\int_{\p\M}\abs{\nabla v}^2\dse \, \dd t\cr
&\quad +\int_0^T\!\!\!\int_\M DN(\nabla v, \nabla\overline{v})\dve \, \dd t
-\frac{1}{2}\int_0^T\!\!\!\int_\M\abs{\nabla v}^2\dive N\,\dve \, \dd t.
\end{align*}
Using the fact $\abs{\nabla v}^2=\abs{\p_\nu v}^2+\abs{\nabla_\tau v}^2=\abs{\p_\nu v}^2$, $x\in\p\M$, where $\nabla_\tau$ is the tangential gradient on $\p\M$, we get
\begin{align}\label{7.26}
\textrm{Re}\,I_2 &=\frac{1}{2}\int_0^T\!\!\!\int_{\p\M}\abs{\p_\nu v}^2\dse \, \dd t +\int_0^T\!\!\!\int_\M DN(\nabla v, \nabla\overline{v})\dve
\, \dd t \\ & \quad -\frac{1}{2}\int_0^T\!\!\!\int_\M\abs{\nabla v}^2\dive N\,\dve \, \dd t.
\end{align}
Finally by (\ref{7.7}) and (\ref{7.13}), we have
\begin{equation}\label{7.27}
\abs{\textrm{Re}\, I_3}\leq \norm{F}^2_{\mathcal{H}}.
\end{equation}
Collecting (\ref{7.27}), (\ref{7.26}), (\ref{7.23}) and (\ref{7.20}), we obtain
\begin{equation}\label{7.28}
\int_0^T\!\!\!\int_{\p\M}\abs{\p_\nu v}^2\dse \dd t\leq C\norm{F}^2_{\mathcal{H}}.
\end{equation}
This completes the proof of (\ref{7.18}).
\end{proof}
\subsection{Non-homogenous boundary condition}
We now turn to the non-homogenous Schr\"odinger problem (\ref{1.2}). \\
Let $\mathcal{H}=L^1(0,T;H^1_0(\M))$ or $\mathcal{H}=H^1_0(0,T;L^2(\M))$.
By $\para{\cdot,\cdot}_{\mathcal{H}',\mathcal{H}}$, we denote the dual pairing between $\mathcal{H}'$ and $\mathcal{H}$.
\begin{definition}
Let $T>0$ and $f\in L^2(\Sigma)$, we say that $u\in\mathcal{H}'$
is a solution of \eqref{1.2} in the transposition sense if for any $F\in \mathcal{H}$ we have
\begin{equation}\label{7.29}
\para{u,\,\overline{F}}_{\mathcal{H}',\mathcal{H}}=\int_0^T\!\!\!\int_{\p\M}
f(t,x)\p_\nu\overline{v}(t,x)\dse \, \dd t
\end{equation}
where $v=v(t,x)$ is the solution of the homogenous boundary value problem
\begin{equation}\label{7.33}
\left\{
\begin{array}{llll}
\para{i\partial_t+\h_{A,V}}v(t,x)=F(t,x)  & \textrm{in }\,\,Q,\cr
v(T,x)=0 & \textrm{in }\,\,\M,\cr
v(t,x)=0 & \textrm{on } \,\, \Sigma.
\end{array}
\right.
\end{equation}
\end{definition}
One gets the following lemma.
\begin{lemma}\label{L7.4}
Let $f\in L^2(\Sigma)$. There exists a unique solution
\begin{equation}\label{7.30}
u\in\mathcal{C}(0,T;H^{-1}(\M))\cap H^{-1}(0,T;L^2(\M))
\end{equation}
defined by transposition, of the problem
\begin{equation}\label{7.31}
\left\{
\begin{array}{llll}
\para{i\partial_t+\Delta_A}u(t,x)=0  & \textrm{in }\,\,Q,\cr
u(0,x)=0 & \textrm{in }\,\,\M,\cr
u(t,x)=f(t,x) & \textrm{on } \,\, \Sigma.
\end{array}
\right.
\end{equation}
Furthermore, there is a constant $C>0$ such that
\begin{equation}\label{7.32}
\norm{u}_{\mathcal{C}(0,T;H^{-1}(\M))}+\norm{u}_{H^{-1}(0,T;L^2(\M))}\leq C\norm{f}_{L^2(\Sigma)}.
\end{equation}
\end{lemma}
\begin{proof}
Let $F\in\mathcal{H}=L^1(0,T;H^1_0(\M))$ or $\mathcal{H}=H_0^1(0,T;L^2(\M))$. Let $v\in\mathcal{C}(0,T;H^1_0(\M))$ solution of the backward
boundary value problem for the Schrödinger equation (\ref{7.33}). By Lemma \ref{L7.3} the mapping $F\mapsto\p_\nu v$ is linear and continuous
from $\mathcal{H}$ to $L^2(\Sigma)$ and there exists $C>0$ such that
\begin{equation}\label{7.34}
\norm{v}_{\mathcal{C}(0,T;H^1_0(\M))}\leq C\norm{F}_{\mathcal{H}}
\end{equation}
and
\begin{equation}\label{7.35}
\norm{\p_\nu v}_{L^2(\Sigma)}\leq C\norm{F}_{\mathcal{H}}.
\end{equation}
We define a linear functional $\ell$ on the linear space $\mathcal{H}$ as follows:
$$
\ell(\overline{F})=\int_0^T\!\!\!\int_{\p\M}f(t,x)\p_\nu\overline{v}(t,x)\ds dt
$$
where $v$ solves (\ref{7.33}). By (\ref{7.35}), we obtain
$$
\abs{\ell(\overline{F})}\leq C\norm{f}_{L^2(\Sigma)}\norm{F}_{\mathcal{H}}.
$$
It is known that any linear bounded functional on the space $\mathcal{H}$ can be written as
$$
\ell(\overline{F})=\para{u,\overline{F}}_{\mathcal{H}',\mathcal{H}}
$$
where $u$ is some element from the space $\mathcal{H}'$. Thus the system (\ref{7.31}) admits a solution $u\in\mathcal{H}'$ in the transposition sense,
which satisfies
$$
\norm{u}_{\mathcal{H}'}\leq C\norm{f}_{L^2(\Sigma)}.
$$
This completes the proof of the Lemma.
\end{proof}
In what follows, we will need the following estimate for non-homogenous elliptic boundary value problem.\\
Let $\psi\in H^{-1}(\M)$ and $\phi\in H^1(\p\M)$. Let $w\in H^1(\M)$ the unique solution of the following boundary value problem
\begin{equation}\label{7.36}
\left\{
\begin{array}{llll}
\Delta_A w =\psi  & \textrm{in }\,\,\M,\cr
w=\phi & \textrm{on } \,\,\p\M,
\end{array}
\right.
\end{equation}
then, by the elliptic regularity (see \cite{[LionsMagenes]}), the following estimate holds true
\begin{equation}\label{7.37}
\norm{w}_{H^1(\M)}\leq C\para{\norm{\psi}_{H^{-1}(\M)}+\norm{\phi}_{H^{1/2}(\p\M)}}.
\end{equation}
\subsection{Proof of Theorem \ref{Th0}}
We proceed to prove Theorem \ref{Th0}. Let $f\in H^{2,1}_0(\Sigma)$ such that $f(0,\cdot)=\p_tf(0,\cdot)=0$ and $u$ solve (\ref{1.2}). Put $w=\p^2_tu$, then
\begin{equation}\label{7.50}
\left\{
\begin{array}{llll}
\para{i\partial_t+\h_{A,V}}w(t,x)=0  & \textrm{in }\,\,Q,\cr
w(0,x)=0 & \textrm{in }\,\,\M,\cr
w(t,x)=\p_t^2f(t,x) & \textrm{on } \,\, \Sigma,
\end{array}
\right.
\end{equation}
Since $\p_t^2f\in L^2(\Sigma)$, by lemma \ref{L7.4}, we get
\begin{equation}\label{7.51}
w\in\mathcal{C}(0,T;H^{-1}(\M))\cap H^{-1}(0,T;L^2(\M)).
\end{equation}
Furthermore there is a constant $C>0$ such that
\begin{equation}\label{7.52}
\norm{w}_{\mathcal{C}(0,T;H^{-1}(\M))}+\norm{u'}_{H^{-1}(0,T;L^2(\M))}\leq C\norm{f}_{H^{2,1}(\Sigma)}.
\end{equation}
Thus (\ref{7.51}) implies the following regularity for $v:=\p_tu$
\begin{align*}
v &\in\mathcal{C}^1(0,T;H^{-1}(\M))\cap \mathcal{C}(0,T;L^2(\M)), \\
\Delta_A v &\in\mathcal{C}(0,T;H^{-1}(\M))\cap H^{-1}(0,T;L^2(\M)).
\end{align*}
Since $\p_tf(t,\cdot)\in H^1(\p\M)$, by the elliptic regularity, we get
$$
v\in \mathcal{C}(0,T;H^1(\M))\cap \mathcal{C}^1(0,T;H^{-1}(\M)).
$$
Moreover there exists $C>0$ such that the following estimates hold true
\begin{equation}\label{7.53}
\norm{v}_{\mathcal{C}^1(0,T;H^{-1}(\M))}+ \norm{\Delta v}_{\mathcal{C}(0,T;H^{-1}(\M))} \leq C\norm{f}_{H^{2,1}(\Sigma)}.
\end{equation}
Using (\ref{7.37}), we find
\begin{equation}\label{7.54}
\norm{v}_{\mathcal{C}^1(0,T;H^{-1}(\M))}+\norm{v}_{\mathcal{C}(0,T;H^{1}(\M))}\leq C\norm{f}_{H^{2,1}(\Sigma)}.
\end{equation}
We deduce the following regularity of the solution $u$
$$
u\in \mathcal{C}^1(0,T;H^1(\M)).
$$
Moreover there exists $C>0$ such that the following estimates hold true
\begin{equation}\label{7.555}
\norm{u}_{\mathcal{C}^1(0,T;H^{1}(\M))} \leq C\norm{f}_{H^{2,1}(\Sigma)}.
\end{equation}
The proof of (\ref{1.8}) is as in Lemma \ref{L7.3}. If one multiplies (\ref{1.2}) by $\seq{N,\nabla\overline{u}}$, the arguments leading to
(\ref{7.20}) give now
\begin{align}\label{7.55}
0=&\int_0^T\!\!\!\int_\M i\p_tu\seq{N,\nabla\overline{u}}\dve \, \dd t+\int_0^T\!\!\!\int_\M\Delta u
\seq{N,\nabla\overline{u}}\dve \, \dd t\cr
&+\int_0^T\!\!\!\int_\M (-2i\seq{A,du}u-i(\delta A)u+\abs{A}^2u+V(x)u)\seq{N,\nabla\overline{u}}\dve \, \dd t =I'_1+I'_2+I'_3,
\end{align}
with
\begin{equation}\label{7.56}
\abs{\textrm{Re}\, I'_1}\leq C_\varepsilon\norm{f}_{H^{2,1}(\Sigma)}^2+\varepsilon\norm{\p_\nu u}^2_{L^2(\Sigma)},
\end{equation}
where we have used (\ref{7.54}) instead of (\ref{7.13})-(\ref{7.7}). Furthermore, we derive from Green's formula
\begin{align}\label{7.57}
\textrm{Re}\,I'_2&=\frac{1}{2}\int_0^T\!\!\!\int_{\p\M}\abs{\p_\nu u}^2\dse \, \dd t  +\int_0^T\!\!\!\int_\M DN(\nabla u,
\nabla\overline{u})\dve \, \dd t\cr
&\quad -\frac{1}{2}\int_0^T\!\!\!\int_\M\!\!\abs{\nabla u}^2\dive N\,\dve \, \dd t-\frac{1}{2}\int_0^T\!\!\!\int_{\p\M}\!\!\abs{\nabla_\tau f}^2\dse
\, \dd t.
\end{align}
This together with
\begin{equation}\label{7.58}
\abs{\textrm{Re}\, I'_3}\leq \norm{f}^2_{H^{2,1}(\Sigma)}
\end{equation}
and (\ref{7.58}), (\ref{7.57}) and (\ref{7.56}) imply
\begin{equation}
\norm{\p_\nu u}_{L^2(\Sigma)}\leq C\norm{f}_{H^{2,1}(\Sigma)},
\end{equation}
where we have used (\ref{7.54}) again. The proof of Theorem \ref{Th0} is now complete.
\section{Some technical results}
\label{sec: appendix}
\setcounter{equation}{0}
\begin{definition}
Suppose that $(\M,\g)$ is a Riemannian manifold. Given a path $\gamma: [a, b]\to \M$, the parallel transport 
$$
J_{(\gamma(a),\gamma(b))}:T_{\gamma(a)}\M\To T_{\gamma(b)}\M,
$$
along $\gamma$ of the tangent vector $X \in T_{\gamma(a)}\M$ is defined as
$$
J_{(\gamma(a),\gamma(b))}(X)=V(b)
$$
where the vector field $V (t) \in T_{\gamma(t)}\M$ is such that
$$\left\{
\begin{array}{lll}
\nabla_{\dot{\gamma}(t)}V(t)=0 & t\in [a,b]\cr
V(a)=X
\end{array}
\right.
$$
that is
$$
V(t)=J_{(\gamma(a),\gamma(t))}(X)
$$
and $\nabla_{\dot{\gamma}}$ is the covariant derivative along $\gamma$. The parallel transport is a linear isometry between
$T_{\gamma(a)}\M$ and $T_{\gamma(b)}\M$.
\end{definition}
\begin{lemma}
Parallel transport is linear, orthogonal, and respects the operations
of reparametrization, inversion and composition:
\begin{eqnarray}
&J_{(\gamma(a),\gamma(b))}\in\mathscr{L}(T_{\gamma(a)}\M,T_{\gamma(b)}\M).\\
&\seq{J_{(\gamma(a),\gamma(b))}(X),J_{(\gamma(a),\gamma(b))}(Y)}=\seq{X,Y},\quad X,Y\in T_{\gamma(a)}\M.\\
&J_{(\gamma(a),\gamma(b))}^{-1}=J_{(\gamma(b),\gamma(a))}\\
&J_{(\gamma(a),\gamma(c))}\circ J_{(\gamma(c),\gamma(b))}=J_{(\gamma(a),\gamma(b))}
\end{eqnarray}
\end{lemma}
For a fixed $x\in\M$ let $v\in T_x\M$. Let $J_{(x,\exp_x\!v)}:T_x\M\To T_{\exp_x\!v}\M$ the parallel transport along the geodesic $\gamma:t\to\exp_x\!tv$, $t\in[0,1]$. We define the Fourier transform on $T_x\M$ as the linear operator $\mathscr{F}:\mathcal{S}'(T_x\M)\To\mathcal{S}'(T^*\M)$ on the space of temporary distribution by
$$
\mathscr{F}(f)(\xi)=\frac{1}{(2\pi)^{n/2}}\int_{T_x\M} e^{-i\seq{\xi,v}}f(\exp_x\!v) dv.
$$
Now, we compute the composition $\I_1^*\I_1$. Let $A\in L^2(\M,T\M)$, by (\ref{2.13}) we have
\begin{eqnarray}
(\I_1^*\I_1(A))_j(x)&=&\int_{S_x\M}\theta^j \widecheck{\I_1(A)} (x,\theta) d\omega_x(\theta)\cr
&=& \int_{S_x\M}\theta^j\I_1(A)(\Phi_{\tau_-(x,\theta)}(x,\theta)) d\omega_x(\theta)
\end{eqnarray}
Since, for $\Sigma_A(x,\theta)=\seq{A^\sharp,\theta}$, we get
\begin{eqnarray}
\I_1(A)(\Phi_{\tau_-(x,\theta)}(x,\theta))&=&\int_0^{\tau_+(\Phi_{\tau_-(x,\theta)}(x,\theta))}\Sigma_A(\Phi_t(\Phi_{\tau_-(x,\theta)}(x,\theta)))dt\cr
&=&\int_0^{\tau_+(x,\theta)-\tau_-(x,\theta)} \Sigma_A(\Phi_{t+\tau_-(x,\theta)}(x,\theta)))dt\cr
&=&\int^{\tau_+(x,\theta)}_{\tau_-(x,\theta)} \Sigma_A(\Phi_{t}(x,\theta)))dt.
\end{eqnarray}
Then
\begin{eqnarray}
(\I_1^*\I_1(A))_j(x)&=&\int_{S_x\M}\theta^j   \int^{\tau_+(x,\theta)}_{\tau_-(x,\theta)} \Sigma_A(\Phi_{t}(x,\theta)))dt d\omega_x(\theta)\cr
&=&2\int_{S_x\M}\theta^j   \int^{\tau_+(x,\theta)}_{0} \Sigma_A(\Phi_{t}(x,\theta)))dt d\omega_x(\theta).
\end{eqnarray}
We denote by $dv_x(\xi)$ the volume form on $T_x\M$ for a fixed $x\in\M$, we consider the following change integration varibales in $T_x\M$ as follows $\xi=t\theta$. Then $dv_x(\xi)=\abs{\xi}^{n-1}\abs{dt\wedge d\omega_x(\theta)}$
$$
(\I_1^*\I_1(A))_j(x)=2\int_{T_x\M}\frac{v^j}{\abs{v}^{n+1}} \Sigma_A(\exp_x\!v,J_{(x,\exp_x\!v)}(v))dv_x.
$$
Since
\begin{eqnarray}
\Sigma_A(\exp_x\!v,J_{(x,\exp_x\!v)}(v))&=&\seq{A(\exp_x\!v)^\sharp,J_{(x,\exp_x\!v)}(v)}\cr
&=&\seq{J_{(\exp_x\!v,x)}A(\exp_x\!v)^\sharp,v}=\sum_{k=1}^n\para{J_{(\exp_x\!v,x)}A(\exp_x\!v)}_k v^k.
\end{eqnarray}
Thus
\begin{equation}\label{B.9}
(\I_1^*\I_1(A))_j(x)=2\sum_{k=1}^n\int_{T_x\M}\frac{v^jv^k}{\abs{v}^{n+1}} \para{J_{(\exp_x\!v,x)}A(\exp_x\!v)}_kdv_x
\end{equation}
we denote by
$$
\varrho_{jk}(x,\xi)=2\mathscr{F}\para{\frac{v^jv^k}{\abs{v}^{n+2}}}(\xi)=2\mathscr{F}^{-1}\para{\frac{v^jv^k}{\abs{v}^{n+2}}}(\xi)
$$
the last equality holds because $\mathscr{F}$ is applied to an evev function. Thus by the inversion formula for the Fourier transform
$$
2\frac{v^jv^k}{\abs{v}^{n+2}}=\frac{1}{(2\pi)^n}\int_{T^*\M} e^{-i\seq{v,\xi}}\varrho_{jk}(x,\xi) d\xi
$$
By (\ref{B.9}) we deduce that
\begin{equation}
(\I_1^*\I_1(A))_j(x)=\frac{1}{(2\pi)^n}\sum_{k=1}^n\int_{T_x\M}\int_{T^*\M} e^{-i\seq{v,\xi}}
\varrho_{jk}(x,\xi)\para{J_{(\exp_x\!v,x)}A(\exp_x\!v)}_kdv_xd\xi
\end{equation}
\section{Smoothness of Distance Function}
Now let's $y\in \p\M_1$ and consider the distance function
$$
\psi :\M_1 \To \R;\quad \psi(x) = d_\g(y,x).
$$
As we have already seen, $\psi$ is a continuous function. However, it is not hard to see
that $\psi$ is not smooth on $\M_1$. In fact, $\psi$ is never smooth at $y$.
\begin{theorem}
The function $\psi$ is smooth on $\M_1\backslash \textrm{Cut}(y)\cup\set{y}$. Moreover, for each $x\in\M_1\backslash \textrm{Cut}(y)\cup\set{y}$, if we let $\gamma_{y,\theta}$
 be the unique normal minimizing geodesic from $y$ to
$x$, then the gradient of $\nabla\psi(x)$ at $x$ is
$$
\nabla\psi(x)=\dot{\gamma}_{y,\theta}(r),\quad r=d_\g(y,x).
$$
\begin{proof}
For each $x \in \M_1\backslash \textrm{Cut}(y)\cup\set{y}$, let $\gamma_{y,\theta}$ be the unique normal minimizing geodesic from 
 $y$ to $x$, $\theta\in S_y\M_1$. Let
$$
A =\set{\ell(\gamma_{y,\theta})\theta,\quad x\in \M_1\backslash \textrm{Cut}(y)\cup\set{y}}.
$$
Then $A\subset T_y\M_1\backslash\{0\}$ is an open set and
$\exp_y : A \to \M_1\backslash \textrm{Cut}(y)\cup\set{y}$ is smooth. Moreover, at each vector in $A$, $\exp_y$ is nonsingular and thus a  local diffeomorphism. Since $\exp_y$ is globallay one-to-one on $A$, it is a diffeomorphism from $A$ to $\M_1\backslash \textrm{Cut}(y)\cup\set{y}$.
It follows that $\exp_y^{-1}:\M_1\backslash\textrm{Cut}(y)\cup\set{y}\To A\subset T_y\M_1\backslash\{0\}$ is smooth. Thus $\psi(x)=\abs{\exp_y^{-1}(x)}$ is smooth on $\M_1\backslash \textrm{Cut}(y)\cup\set{y}$. To calculate its gradient at $x$, we choose any $X\in T_x\M_1$ and let $\sigma(s)$ be a smooth  curve in $\M_1\backslash\textrm{Cut}(y)\cup\{y\}$ tangent to $X$ at $x=\sigma(0)$.

Now we consider the variation of  $\gamma_{y,\theta}$ so that $V(s,\cdot)$ be the unique minimizing geodesic from $y$ to $\sigma(s)$. Observe that the variation field vector of this variation at the point $x$ is exactly $X$. So according to the first variation formula,
$$
X(\psi) =\frac{d}{ds}\psi(\sigma(s))_{|s=0}=\frac{d}{ds}\ell(V(r,s))_{|s=0}= \seq{X,\dot{\gamma}_{y,\theta}(r)}.
$$
It follows that $\nabla\psi(x)=\dot{\gamma}_{y,\theta}(r)$.
\end{proof}
\end{theorem}

\end{document}